\documentclass[smallextended,envcountsect,]{svjour3blank}

\usepackage{lmodern}
\smartqed
\usepackage{graphicx}

\newcommand{\hidefigures}[2]{#1} 

\usepackage[inline]{enumitem}  
\usepackage[utf8]{inputenc}
\usepackage{amsmath}  
\usepackage{amssymb}
\usepackage{bbm}    
\usepackage{bm}
\usepackage{accents}
\usepackage{mathrsfs}
\usepackage{mathtools}
\usepackage{fixmath}
\usepackage{tikz,pgfplots}
\usetikzlibrary{intersections,positioning,calc}
\usepgfplotslibrary{fillbetween}
\usepackage{microtype}
\usepackage{hyperref}
\usepackage{comment}
\hidefigures{}{\excludecomment{figure}
}%

\usepackage[sharp]{easylist}

\hypersetup{
colorlinks=true,
linktocpage=true,
pdfstartview=FitH,
breaklinks=true,
pdfpagemode=UseNone,
pageanchor=true,
pdfpagemode=UseOutlines,
plainpages=false,
bookmarksnumbered,
bookmarksopen=false,
bookmarksopenlevel=1,
hypertexnames=true,
pdfhighlight=/O,
urlcolor=black,
linkcolor=black,
citecolor=black,
pdftitle={},
pdfauthor={},
pdfsubject={},
pdfkeywords={},
pdfcreator={pdfLaTeX},
pdfproducer={LaTeX with hyperref}
}
 
\usepackage[capitalise]{cleveref}
\usepackage{complexity}
\usepackage{upgreek} 
\usepackage{xcolor}
\usepackage{stmaryrd}

\makeatletter
\newcommand{\pushright}[1]{\ifmeasuring@#1\else\omit\hfill$\displaystyle#1$\fi\ignorespaces}
\newcommand{\pushleft}[1]{\ifmeasuring@#1\else\omit$\displaystyle#1$\hfill\fi\ignorespaces}
\makeatother

\newcommand{\obsolete}[1]{}
\newcommand{\nocolsep}{\arraycolsep=1.4pt\def\arraystretch{1}}
\newcommand{\refereq}[2]{\overset{\text{\tiny{#1}}}{#2}}%

\DeclareMathOperator\symdif{\triangle}%

\newcommand{\Integer}{\mathbb{Z}}
\newcommand{\Real}{\mathbb{R}}

\DeclareMathOperator{\supp}{supp}

\newcommand{\cstB}{B}

\newcommand{\cstD}{D}
\newcommand{\cstDprime}[1]{\cstD^\prime}

\newcommand{\xvar}{x}
\newcommand{\yvar}{y}
\newcommand{\zvar}{z}

\newcommand{\altxvar}{\xvar^\prime}
\newcommand{\altaltxvar}{\xvar^{\prime\prime}}

\newcommand{\altzvar}{\zvar^\prime}

\newcommand{\yvari}[1]{\yvar_{#1}}

\newcommand{\distfunction}{d}

\newcommand{\dist}[1]{\distfunction(#1)}

\newcommand{\xmeasurefunction}{\mu}
\newcommand{\xmeasure}[1]{\xmeasurefunction(#1)}

\newcommand{\xMeasure}[1]{\xmeasurefunction\left(#1\right)}
\newcommand{\bigxmeasure}[1]{\xmeasurefunction\big(#1\big)}

\newcommand{\ymeasurefunction}{\nu}

\newcommand{\ymeasurei}[1]{\ymeasurefunction_{#1}}

\newcommand{\pricevalueset}{P}

\newcommand{\profitfunction}{\mathbb{P}}%
\newcommand{\profit}[1]{\profitfunction(#1)}%
\newcommand{\profitdeltafunction}[1]{\profitfunction^{#1}}%
\newcommand{\profitdelta}[2]{\profitdeltafunction{#1}(#2)}%
\newcommand{\pricefunction}{p}

\newcommand{\pricei}[1]{\pricefunction_{#1}}

\newcommand{\pricevec}{\bm{\pricefunction}}
\newcommand{\profile}{\pricevec}
\newcommand{\altprofile}{\tilde{\profile}}

\newcommand{\altpricefunction}{\pricefunction^\prime}

\newcommand{\altpricei}[1]{\altpricefunction_{#1}}

\newcommand{\altaltprofile}{q}
\newcommand{\altaltprice}{\altaltprofile}
\newcommand{\altaltpricei}[1]{\altprice_{#1}}

\newcommand{\pricemin}{\pricefunction^{\min}}

\newcommand{\pricemax}{\pricefunction^{\max}}

\newcommand{\Smooth}{\lambda}%
\newcommand{\lipschitz}{L}%
\newcommand{\stronglyconvex}{M}%

\newcommand{\xmeasuremax}{\bar{\xmeasurefunction}}
\newcommand{\ubar}[1]{\text{\b{$#1$}}}
\newcommand{\xmeasuremin}{\ubar{\xmeasurefunction}}

\newcommand{\IdistLipschitz}{\lipschitz_{\distfunction}}
\newcommand{\Idist}{I_{\distfunction}}

\renewcommand{\altaltprice}{q}%
\renewcommand{\altaltpricei}[1]{\altaltprofile_{#1}}%
\renewcommand{\altaltprofile}{\profile'}%

\newcommand{\pricemdeltafunction}[1]{\pricefunction^{#1}}
\newcommand{\altpricemdeltafunction}[1]{\altpricefunction^{#1}}

\renewcommand{\xvar}{s}
\renewcommand{\yvar}{t}
\renewcommand{\pricevec}{p}

\renewcommand{\altpricefunction}{\tilde{p}}

\newcommand{\latchfunction}[2]{\xvar_{{#1}{#2}}}%
\newcommand{\latch}[3]{\latchfunction{#1}{#2}({#3})}%

\newcommand{\Intervalsymbol}{I}
\newcommand{\fullIntervalsymbol}{J}
\newcommand{\Intervalfunction}{\Intervalsymbol}
\newcommand{\Intervalifunction}[1]{\Intervalfunction_{#1}}
\newcommand{\Intervali}[2]{\Intervalifunction{#1}({#2})}
\newcommand{\bigIntervali}[2]{\Intervalifunction{#1}\big({#2}\big)}
\newcommand{\fullIntervalfunction}{\fullIntervalsymbol}
\newcommand{\fullIntervalifunction}[1]{\fullIntervalfunction_{#1}}
\newcommand{\fullIntervali}[2]{\fullIntervalifunction{#1}({#2})}%

\newcommand{\slot}{k}
\newcommand{\altslot}{i}
\newcommand{\altaltslot}{j}

\newcommand{\spongeslot}{\slot^\star}

\newcommand{\profilemdelta}[1]{\pricemdeltafunction{#1}}

\newcommand{\pricemdeltai}[2]{\pricemdeltafunction{#1}_{#2}}
\newcommand{\altprofilemdelta}[1]{\altpricemdeltafunction{#1}}

\newcommand{\vanishingfunction}{h}%
\newcommand{\vanishing}[1]{\vanishingfunction(#1)}%

\newcommand{\heavisidemvectorsymbol}{v}
\newcommand{\heavisidemvectori}[1]{\heavisidemvectorsymbol_{#1}}
\newcommand{\heavisidepvectorsymbol}{w}
\newcommand{\heavisidepvectori}[1]{\heavisidepvectorsymbol_{#1}}
\newcommand{\deltamax}{\delta^{\max}}
\newcommand{\MCspongefunction}{N}%
\newcommand{\MCtool}{K}%
\newcommand{\MCwave}{T}%
\newcommand{\loadmax}{\tau}

\newcommand{\sublevelsetsymbol}{\mathcal{L}}%
\newcommand{\sublevelsetfunction}{\sublevelsetsymbol}%
\newcommand{\sublevelsetifunction}[1]{\sublevelsetfunction_{#1}}%
\newcommand{\sublevelseti}[2]{\sublevelsetifunction{#1}({#2})}%

\newcommand{\csth}{h}

\newcommand{\cstDeltafunction}{f}%
\newcommand{\cstDelta}[1]{\cstDeltafunction({#1})}

\newcommand{\cstgain}{\gamma}

\newcommand{\xivariable}{\xi}

\newcommand{\altpriceifunction}[1]{\pi}

\renewcommand{\profitfunction}{R}

\newcommand{\pricemindeltafunction}{\pricemin}%
\newcommand{\pricemindelta}[1]{\pricemindeltafunction}

\renewcommand{\cstB}{B}
\newcommand{\cstBprime}[1]{\cstB^\prime}

\DeclareMathOperator{\OPTsymbol}{Opt}
\newcommand{\OPT}[1]{\OPTsymbol({#1})}
\newcommand{\OPTdeltafunction}[1]{\OPTsymbol^{#1}}
\newcommand{\OPTdelta}[2]{\OPTdeltafunction{#1}({#2})}

\colorlet{pricecolor}{white}
\colorlet{diffpricecolor}{white}
\colorlet{xvarcolor}{white}
\colorlet{ecolor}{white}
\colorlet{scolor}{white}
\colorlet{presentcolor}{cyan!70!blue}
\colorlet{pastcolor}{yellow!10!lime}

\colorlet{linepathcolor}{white}

\usetikzlibrary{shapes}
\usetikzlibrary{arrows.meta}

\spnewtheorem{fact}{Fact}[section]{\bfseries}{\itshape}\crefname{fact}{Fact}{Facts} 

\renewcommand{\geq}{\geqslant}
\renewcommand{\leq}{\leqslant}

\def\ds{\displaystyle}

\usepackage[colorinlistoftodos,bordercolor=orange,backgroundcolor=orange!20,linecolor=orange,textsize=scriptsize,obeyFinal]{todonotes} 
\usepackage{acronym}		

\newacro{OT}{optimal transport}
\newacro{DP}{dynamic programming}
\newacro{PTAS}{polynomial-time approximation scheme}
\newacro{FPTAS}{fully polynomial-time approximation scheme}
\newacro{ATP}{available to promise}
\newacro{CTP}{capable to promise}
\newacro{PTP}{profitable to promise}
\newacro{EDF}{\'{E}lectricit\'{e} de France}
\newacro{AHD}{attended home delivery}
\newacro{MNL}{multinomial logit}

\definecolor{tealblue}{HTML}{007050}
\definecolor{MyBlue}{HTML}{000070}
\definecolor{MyGreen}{HTML}{007070}
\definecolor{changes}{HTML}{0050B0}
\definecolor{maroon}{HTML}{c00050}

\newcommand{\old}[1]{}
\definecolor{Cloud}{HTML}{F5F5F5}
\definecolor{FMcolor}{HTML}{CCEEFF}

\newcommand{\OBoff}[1]{}

\begin{document}

\title{An Exact Method for a Problem of Time Slot Pricing}


\author{Olivier Bilenne \and Frédéric Meunier}

\institute{%
Olivier Bilenne \at
    Conservatoire national des arts et m\'etiers, C\'EDRIC\\ 
    Paris, France\\
    olivier-stephane.bilenne@lecnam.net
\and
Frédéric Meunier \at
    CERMICS, \'Ecole nationale des ponts et chaussées\\
    France\\
    frederic.meunier@enpc.fr
%
}

\date{}

\maketitle

\begin{abstract}
A company provides a service at different time slots, each slot being endowed with a capacity. A non-atomic population of users is willing to purchase this service. The population is modeled as a continuous measure over the preferred times. Every user looks at the time slot that minimizes the sum of the price assigned by the company to this time slot and the distance to their preferred time. If this sum is non-negative, then the user chooses this time slot for getting the service. If this sum is positive, then the user rejects the service.

We address the problem of finding prices that ensure that the volume of users choosing each time slot is below capacity, while maximizing the revenue of the company. For the case where the distance function is convex, we propose an exact algorithm for solving this problem in time $O(n^3|P|^3)$, where $P$ is the set of possible prices and $n$ is the number of time slots. For the case where the prices can be any real numbers, this algorithm can also be used to find asymptotically optimal solutions in polynomial time under mild extra assumptions on the distance function and the measure modeling the population.
\end{abstract}

\section{Introduction}

\subsection{Motivation} Consider a company offering a service to a large population of users at different time slots, each with a limited capacity. The company can assign a distinct price to each slot, aiming at simultaneously keeping the number of users served on each time slot below some threshold, and maximizing its profit. This is exactly the problem addressed in the present paper.

This work results from a collaboration with \ac{EDF}, the French electricity producer. In anticipation of the growing popularity of electric vehicles, EDF has explored various aspects of this topic. One of them is whether carefully designed pricing strategies could help smooth out demand at charging stations. More precisely, assuming a perfect forecast of the demand, is it possible to adjust the time slot prices so as to incentivize users to change the time of their recharge and to keep the number of users always below some threshold, while maximizing the revenue? More generally, the question of smoothing demand at charging stations belongs to the broader task of efficiently managing electricity consumption through dynamic pricing schemes.

This question actually looks relevant across diverse domains where users express preferences for receiving services at specific times, and their decisions are influenced by price and inconvenience incurred due to deviations from their preferred timing. Among these we find for instance ridesharing services, public transportation systems, streaming platforms, etc. along with all domains where pricing strategies can be used to match supply and demand while maximizing revenue.
 
In this work, we introduce a modeling approach for problems of this nature and an efficient (polynomial) strategy to address them. 

\subsection{Problem Formulation}
A company is providing a service at different time slots and a population of users is willing to benefit from this service. Each user is characterized by a preferred time $s$. The preferred times are modeled by a finite absolutely continuous measure $\mu$ over $\Real$ (which throughout the paper is endowed with the Lebesgue measure): for a measurable subset $A$ of $\Real$, the quantity $\mu(A)$ is the volume of users whose preferred time for being served is in $A$. The cost incurred by a user for being served at time $t$ while preferring time $s$ is modeled as $d(s-t)$, where $d$ is a strictly convex function (and thus continuous) with its unique minimum at $0$. 

Service availability is restricted to~$n$ distinct time slots $1, 2, \dots, n$. Each time slot~$j$ is characterized by a time of service~$t_j$ and a limited capacity~$\nu_j$: at time $t_j$, there is a maximum volume $\nu_j \in \Real_+$ of users that can be served.

The company assigns each time slot~$j$ a price $p_j$, taken from a given set $P$ of possible prices. The users are free to choose the time slot~$j$ over which they are served, or to choose not to be served at all. Yet, the assumption for the decision process is the following. Each user with preferred time $s$ considers a time slot~$j$ that minimizes the total cost $d(s-t_j) + p_j$: if the minimal value of the total cost is non-positive, then the user chooses to be served at time $t_j$; otherwise, the user chooses not to be served. 

The objective for the company is to maximize its revenue defined as the integral of the price paid by the users getting the service. We formulate the problem as the following mathematical program:
\begin{equation}\label{upper-level}\tag{U}
\begin{array}{r@{\qquad}l@{\quad}l}
\underset{\hspace{-10mm} p_1,\ldots,p_n\hspace{-10mm}}{\text{maximize}} & \ds{\sum_{j=1}^n p_j} \mu(I_j(p))
\\
\text{subject to} 
 & \mu(I_j(p)) \leq \nu_j  & \forall j \in [n]
 \\ & p_j \in P & \forall j \in [n] \, ,
\end{array}
\end{equation}
where $I_j(p) = \bigl\{s\in\Real \colon d(s-t_j)+p_j \leq \min(0,d(s-t_k)+p_k) \,\, \forall k \in [n]\bigl\}$ is formed by the users choosing time slot~$j$ for getting the service, and~$\mu(I_j(p))$ forms the service load at~$j$.  We denote by~$\OPT{P}$ the optimal value of the revenue, with an emphasis on its dependency on the price set~$P$. 

Formally, this problem is {\em bilevel}: there is the optimization problem of the company (upper level), but the users, by choosing a time slot or deciding not to be served, are also solving their own elementary optimization problem (lower level). Regarding \eqref{upper-level} as a bilevel program, it is not difficult to see that optimistic and pessimistic interpretations of the problem
---which would respectively assume cooperation and non-cooperation of the users with the company---coincide, due to the 
strict convexity of the fonction~$d$ and  the non-atomic nature of the measure $\mu$. This ensures that the interiors of the sets~$I_1,\dots,I_n$ do not overlap and that problem~\eqref{upper-level} is well formulated. 
See, e.g.,~\cite{colson2007overview} for further discussions about bilevel programming. 

We consider two special cases of the problem:
\begin{itemize}
    \item the case where $P$ is finite.
    \item the case where $P$ is $\Real$, $\mu$ has bounded support, and $d$ is strongly convex.
\end{itemize}

\begin{remark}
The users actually solve what is called in optimal transport theory a {\em semi-discrete transport problem} (see, e.g.,~\cite{santambrogio15}). It is well-known that the dual variables of such problems provide prices that lead to an ``automatic'' satisfaction of the capacity constraints. However, the presence of the revenue as an objective function makes unlikely that straightforward adaptations of results from optimal transport theory could lead to optimal solutions of~\eqref{upper-level}.
\end{remark}

\subsection{Contributions}

Our contributions are mainly algorithmic. We assume that the following operations take constant time: computing the value of $d$ at any point;
inverting and minimizing functions of the form $s\mapsto d(s)$ or $s\mapsto d(s)-d(s-t)$;
computing the value of $\mu$ on any interval. 
It is worth noting that the strict convexity of function~$d$ lends plausibility to the computational assumption; 
in particular, $s \mapsto d(s) - d(s-t)$ is increasing (see \cref{fact:increasing} below).
Under this computation assumption, we prove two theorems. The first one is concerned with finite price sets.

\begin{theorem}\label{thm:semi-disc}
When $P$ is finite, the problem can be solved in time $O(n^3|P|^3)$.
\end{theorem}
Our second main result addresses the case when $P$ is $\Real$.
The main message is that, under mild assumptions, we can efficiently find solutions within arbitrary optimality gaps. 
%
Suppose that $\mu$ has a bounded support, and for $\delta>0$ consider the countable set~$\delta\Integer$. 
Since~$\delta\Integer$ is a subset of~$\Real$, we have $\OPT{\delta\Integer} \leq \OPT{\Real}$. 
Using \cref{thm:semi-disc}, we show that the computation of the lower bound $\operatorname{LB}(\delta)=\OPT{\delta\Integer}$ can be done in time $O(\frac 1 {\delta^3} n^3)$ (see \cref{sec:lowerbound}), and the next theorem claims that a close upper bound can be computed with the same time complexity at the price of a small extra assumption (see \cref{sec:upperbound}). This upper bound only depends on the support of~$\xmeasurefunction$, the lower and upper bounds on~$\xmeasurefunction$, the function~$\distfunction$, and the times~$\yvari{1},\dots,\yvari{n}$, and is independent of the values given to the measures~$\xmeasurefunction$ and~$\ymeasurefunction$.

\begin{theorem}\label{thm:cont}
Assume that~$\distfunction$ is strongly convex and the density of~$\xmeasurefunction$ is interval supported and lower- and upper-bounded above zero. Then an upper bound~$
\operatorname{UB}(\delta)
$ on $\OPT{\Real}$ can be computed in time $O(\frac 1 {\delta^3} n^3)$, such that $
\operatorname{UB}(\delta)-\operatorname{LB}(\delta)
= O(\delta^{1/8})$.
\end{theorem}

We close this section by emphasizing that our algorithms are conceptually simple and easy to implement; the length of the paper is mostly due to the proofs, which require some work.


\subsection{Literature Review}
\label{sec:stateoftheart}

Many optimization problems focusing on time slot pricing are addressed from the perspective of revenue management, for which a rich literature exists (see, e.g.,~\cite{van2005introduction} for an introduction to revenue management).
Several studies have also examined the question of time slot pricing, in particular in the context of \ac{AHD} where, usually, the users are assumed to enter the time slot booking systems sequentially according to a stochastic process 
and the revenue maximization problem is solved heuristically (see, e.g.,~\cite{akkerman22} and references therein). A common trend in \ac{AHD} is to develop online solutions based on stochastic dynamic programming, which break down the pricing problem into sub-problems solved recursively for the user arrival times~\cite{asdemir2009dynamic,lebedev2021dynamic,strauss2021dynamic,yang17}.
Other such studies explore the specific context of charging stations for electric vehicles from diverse angles. 
In~\cite{kazemtarghi24}, for instance, the pricing problem at charging stations is modeled as a linear program for revenue maximisation and solved by an interior point method. In~\cite{xiong16}, time slot pricing is seen as a congestion game for social cost minimization. Bayesian stochastic models have also been considered for revenue maximization at charging stations and load balancing, in combination with reinforcement learning~\cite{zhang22} or with genetic algorithms~\cite{kalakanti25,moghaddam20}. A genetic algorithm is used in~\cite{liu23} as well to solve a pricing problem modeled as multi-objective optimization. An approach in spirit close to our work is found in~\cite{anjos2024integrated}, where the time slot pricing of charging stations is formulated as a bilevel mathematical program for the service provider (upper-level problem) and for the users (lower-level problem), and solved heuristically by integer linear programming after single-level reformulation. 

In addition to the diversity of settings and methods of solution, the litterature on time slot pricing also differs in the way the users are assumed to choose their favourite times slots. The conventions typically used in the literature include rank-based models~\cite{akkerman22}, preference lists with maximal price thresholds~\cite{anjos2024integrated}, and probabilistic models such as the \acl{MNL}~\cite{asdemir2009dynamic,lebedev2021dynamic,strauss2021dynamic,yang17}.
In this paper, the preferences of each user are derived from the distances to the time slots of their ideal time for getting the service. Under this modeling approach, the lower-level problem of time slot selection is reminiscent of the 
optimal transport problem on the real line. This analogy lends the time slot pricing problem the optimal structure of dynamic programming, allowing for exact pricing algorithms as computationally attractive as the online heuristics.

As already pointed out, the idea of using pricing to adjust demand to capacity while maximizing  revenue is especially present for electricity consumption; see, e.g., two surveys on this topic~\cite{crew95,chinhuihao24}. However, we are unaware of any research that addresses the problem examined in this work. In fact, papers proposing exact polynomial algorithms for practically relevant problems related to demand adjustment and revenue maximization---our motivation---are scarce in the literature.


%
\section{Preliminary Results}
\label{sec:preliminaries}

We start with a natural fact in convex analysis, which we state as an observation for future references.

\begin{fact}\label{fact:increasing}
   The map $s \mapsto d(s-t_j) - d(s-t_k)$ is increasing for $j < k$.
\end{fact} 
%
This a direct consequence of the strict convexity of~$d$. Alternatively, the statement follows from the monotonicity of secant slopes of convex functions; see, e.g., \cite[Section~4.1]{bertsekas2003convex}. 

It is convenient in the problem~\eqref{upper-level} to rewrite the sets~$\Intervali{j}{\profile}$ as
\begin{equation}
\label{intersections:a}
\Intervali{j}{\profile} 
=
\fullIntervali{j}{\profile} 
\cap
\sublevelseti{j}{\pricei{j}} 
\quad\forall \profile\in P^n,\ j\in[n]
\, ,
\end{equation}
where $
\fullIntervali{j}{\profile} 
=
\{\xvar\in\Real \colon \dist{\xvar-\yvari{j}} +  \pricei{j} \leq \dist{\xvar-\yvari{k}} + \pricei{k}   \,\, \forall  k \in [n] \}
$,
and
$
 \sublevelseti{j}{\altaltprice} 
 = \{ \xvar \in \Real\colon\dist{\xvar-\yvari{j}}+ \altaltprice \leq 0 \}
$ denotes the sublevel set of~$\xvar\mapsto\dist{\xvar-\yvari{j}}$.
The upcoming result  is well known in the context of optimal transport; see, e.g., \cite[Theorem 2.9]{santambrogio15} which, together with~\cite[Theorem 2.5]{santambrogio15}, characterizes the solutions of a convex optimal transport problem.

\begin{lemma}[Monotonicity]\label{lemma:monotonicity}
Let $\profile
\in  \pricevalueset^n$ and $\xvar,\altxvar \in \Real$. If $\xvar<\altxvar$, then $k \leq k'$ for all $k,k'\in[n]$ such that~$\xvar\in\fullIntervali{k}{\profile}$ and~$\altxvar\in\fullIntervali{k'}{\profile}$.
\end{lemma}

\begin{proof}
Let $k,k'\in[n]$, $\xvar\in\fullIntervali{k}{\profile}$, $\altxvar\in\fullIntervali{k'}{\profile}$, and assume that $\xvar<\altxvar$. 
By definition of~$\fullIntervali{k}{\profile}$ and~$\fullIntervali{k'}{\profile}$, we have
$
\dist{\xvar-\yvari{k}}+\pricei{k}\leq \dist{\xvar-\yvari{k'}}+\pricei{k'}
$ and $
\dist{\altxvar-\yvari{k'}}+\pricei{k'} \leq \dist{\altxvar-\yvari{k}}+\pricei{k}
$.
Adding up the previous two inequalities yields
$
\dist{\xvar-\yvari{k}}+\dist{\altxvar-\yvari{k'}}\leq \dist{\xvar-\yvari{k'}}+ \dist{\altxvar-\yvari{k}}
$, 
which rewrites as
$
\dist{\xvar-\yvari{k}} - \dist{\xvar-\yvari{k'}} \leq  \dist{\altxvar-\yvari{k}} - \dist{\altxvar-\yvari{k'}}
$. It follows from \cref{fact:increasing} that $k \leq k'$.
\qed\end{proof}
\begin{remark}\label{remark:monocity}
 The last inequality in the proof can be interpreted as a continuous analogue of the Monge property for matrices (see, e.g., \cite{burkard96}), which in discrete optimization implies monotone optimal assignments.
\end{remark}

\cref{prop:intervals} shows that the sets $ \Intervali{j}{\profile}  $, $\fullIntervali{j}{\profile}$, $ \sublevelseti{j}{q} $ follow a particular arrangement. 
Its proof relies on~\cref{lemma:monotonicity}. Notice that any of these sets can be empty.

\begin{proposition}\label{prop:intervals}
    The sets $ \Intervali{j}{\profile}  $, $\fullIntervali{j}{\profile}$, $ \sublevelseti{j}{q} $ are closed intervals for every $j \in [n]$ and every $\profile \in P^n$ and $q\in P$. 
    Moreover, for every $\profile \in P^n$, the sets $\fullIntervali{1}{\profile},\dots,\fullIntervali{n}{\profile}$ cover~$\Real$, have pairwise disjoint interiors, and are ordered by increasing index~$j$. 
\end{proposition}

\begin{proof}
    The sets $ \Intervali{j}{\profile}  $, $\fullIntervali{j}{\profile}$, $ \sublevelseti{j}{q} $ are closed as they contain all their limit points. 
    For each $k\in[n]$, the set $\{\xvar\in\Real\colon \dist{\xvar-\yvari{j}}+\pricei{j} \leq \dist{\xvar-\yvari{k}}+\pricei{k} \}$ is an 
    interval by the monotonicity of $\xvar\mapsto\dist{\xvar-\yvari{j}}-\dist{\xvar-\yvari{k}}$ (\cref{fact:increasing}).
    Hence, $\fullIntervali{j}{\profile}$ is an 
    interval as the intersection of all these sets for $k$ ranging over $[n]$.
    The set~$ \sublevelseti{j}{\profile} $ is an 
    interval as the sublevel set of a convex function. 
    It follows from~\eqref{intersections:a} that~$\Intervali{j}{\profile}$ is an 
    interval, as intersection of two 
    intervals. 

    Further, $\xvar\in\Real$ belongs to~$\fullIntervali{j}{\profile}$ whenever~$j$ minimizes $\dist{\xvar-\yvari{j}}+\pricei{j}$ over~$[n]$.
    Since every finite set has minimal elements, at least one such~$j$ exists, and therefore the sets $\fullIntervali{1}{\profile},\dots,\fullIntervali{n}{\profile}$ cover~$\Real$.
    The last two claims are consequences of \cref{lemma:monotonicity}.
\qed\end{proof}
See \cref{fig:intervals} for an illustration of \cref{prop:intervals} and its relation with the service loads at the time slots.

\begin{figure}[ht]
\begin{center}
\includegraphics[]{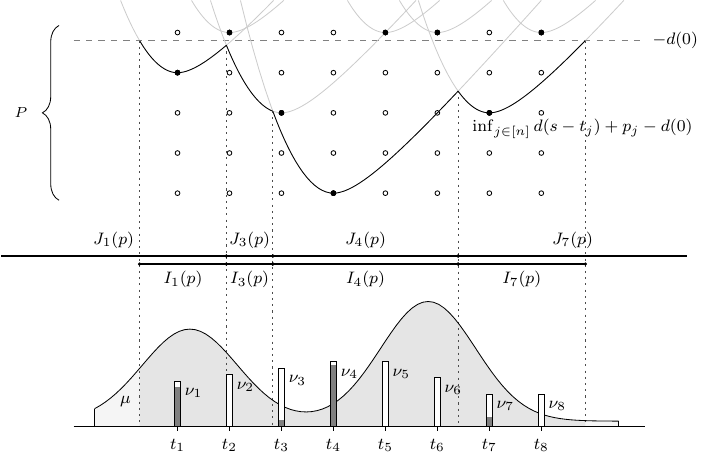}
\end{center}
\caption{\label{fig:intervals}
The black dots form an example of a feasible price profile when $n=8$ and~$P$ has five elements. 
The lower envelope of the functions $\dist{\xvar-\yvari{j}}+\pricei{j}$ determines the intervals~$\fullIntervali{j}{\profile}$, which are ordered as in \cref{prop:intervals}. Its nonpositivity further delimits the intervals~$\Intervali{j}{\profile}$. All time slots with prices above the threshold~$-\dist{0}$ are therefore avoided.
The solution is feasible as the service loads $\xmeasure{\Intervali{j}{\profile}}$ are less than the capacities~$\ymeasurei{j}$.
}
\end{figure}

\section{Algorithm when \texorpdfstring{$P$}{\textit{P}} Is Finite}\label{sec:algoPfinite}

The purpose of this section is to prove~\cref{thm:semi-disc} and to explicitly describe the algorithm guaranteed by this theorem. This is done in a first subsection, assuming two lemmas, which are eventually proved in a second subsection.

\subsection{Description of the Algorithm and Proof of~\cref{thm:semi-disc}}\label{subsec:algoPfinite}
\newcommand{\priceleft}{r}
\newcommand{\pricecentral}{q}
\newcommand{\priceright}{r'}
The algorithm for solving the problem when $P$ is finite uses a directed graph. 
To ease the description of this graph, we will say that a time slot~$\ell$ is {\em covered} by a pair $(i,\pricecentral) \in [n] \times P$ at $s \in \Real$ if 
\[
d(s-t_i) + \pricecentral \leq d(s-t_{\ell}) + \max P .
\]
By extension, we consider that $\ell\in [n]$ is covered by~$(i,\pricecentral)$ at $-\infty$ if
$\lim\nolimits_{s'\rightarrow -\infty} d(s'-t_i) - d(s'-t_{\ell}) \leq \max P - \pricecentral$, and at~$+\infty$ if
$\lim\nolimits_{s'\rightarrow +\infty} d(s'-t_i) - d(s'-t_{\ell}) \leq \max P - \pricecentral$.
These limits are well defined by \cref{fact:increasing}.
From~\cref{fact:increasing} we also know that, for~$i,j\in[n]$ with $i<j$ and $\pricecentral\in P$, the equation $\dist{\xvar-\yvari{i}}-\dist{\xvar-\yvari{j}}=\pricecentral$ has at most one solution. We write this solution~$\latch{i}{j}{\pricecentral}$, where $\latchfunction{i}{j}$ is the inverse of $\xvar\mapsto\dist{\xvar-\yvari{i}}-\dist{\xvar-\yvari{j}}$, with the convention $\latch{i}{j}{\pricecentral}=-\infty$ when $\dist{\xvar-\yvari{i}}-\dist{\xvar-\yvari{j}}>\pricecentral$ for all $\xvar\in\Real$, and 
$\latch{i}{j}{\pricecentral}=+\infty$ when $\dist{\xvar-\yvari{i}}-\dist{\xvar-\yvari{j}}<\pricecentral$ for all $\xvar\in\Real$.
Let $D=(V,A)$ be the directed digraph with $V \coloneqq ([n]\times P ) \cup\{(0,\cdot),(n+1,\cdot)\}$ and with $A \coloneqq A_1 \cup A_2 \cup A_3$, where
\begin{align*}
    A_1 & \coloneqq \left\{ \bigl((0,\cdot),(j,\pricecentral)\bigl) \colon \text{every $\ell$ in $\{1,\ldots,j\}$  is covered by $(j,\pricecentral)$ at $-\infty$}\right\} \\[1ex] 
    A_2 &  \coloneqq \left\{ \bigl((i,\priceleft),(j,\pricecentral)\bigl)\colon i<j, \; \latch{i}{j}{\pricecentral-\priceleft}\neq \pm\infty \; \text{and every $\ell$ in $\{i,\ldots,j\}$ is 
    } \right. 
    \\[-1ex] & \hspace{45mm}  \left. \text{covered by $(i,\priceleft)$ 
    and by $(j,\pricecentral)$ 
    at $\latch{i}{j}{\pricecentral-\priceleft}$}\right\}  \\
    A_3 & \coloneqq \left\{ \bigl((j,\pricecentral),(n+1,\cdot)\bigl) \colon \text{every $\ell$ in $\{j,\ldots,n\}$ is covered by $(j,\pricecentral)$ at $+\infty$}\right\} \, .
\end{align*}
In $D$, the vertices $(0,\cdot)$ and~$(n+1,\cdot)$ are respectively seen as the source and the sink. To each pair of consecutive arcs $a,a'$, we assign a \emph{reward}~$w(a,a')$ as follows.

\smallskip

\begin{easylist}\ListProperties(Style1*=\scshape$\bullet$, Hide1=2, Space1*=0.3cm)

# \label{A1} 
{\em 
When $a = \bigl((0,\cdot),(j,\pricecentral)\bigl)\in A_1$ and $a' =  \bigl((j,\pricecentral),(k,\priceright)\bigl)\in A_2$.} 
Then set
\[
w(a,a') \coloneqq \left\{\begin{array}{ll}\pricecentral v(a,a') & \text{if $v(a,a')\leq \nu_j$,} \\ -\infty & \text{otherwise,}
\end{array}\right.
\]
where
$ 
v(a,a') \coloneqq \mu\bigl(\{s\in(-\infty,\latch{j}{k}{\priceright-\pricecentral}]\colon  d(s-t_j) + \pricecentral\leq 0\}\bigl)
$. 

# \label{A2} 
{\em When $a =  \bigl((i,\priceleft),(j,\pricecentral)\bigl)\in A_2$ and $a'=\bigl((j,\pricecentral),(k,\priceright)\bigl)\in A_2$.} 
Then set 
\[
w(a,a') \coloneqq \left\{\begin{array}{ll}\pricecentral v(a,a') & \text{if $\latch{i}{j}{\pricecentral-\priceleft}\leq \latch{j}{k}{\priceright-\pricecentral}$ and $v(a,a')\leq \nu_j$,} \\ -\infty & \text{otherwise,}
\end{array}\right.
\]
where
$ 
v(a,a') = \mu\bigl(\{s \in [\latch{i}{j}{\pricecentral-\priceleft},\latch{j}{k}{\priceright-\pricecentral}] \colon  d(s-t_j) + \pricecentral\leq 0 \} \bigl)
$. 

# \label{A3} 
{\em When $a = \bigl((i,\priceleft),(j,\pricecentral)\bigl)\in A_2$ and $a' = \bigl((j,\pricecentral),(n+1,\cdot)\bigl)\in A_3$.} 
Then set 
\[
w(a,a') \coloneqq \left\{\begin{array}{ll}\pricecentral v(a,a') & \text{if $v(a,a')\leq \nu_j$,} \\ -\infty & \text{otherwise,}
\end{array}\right.
\]
where
$ 
v(a,a') \coloneqq \mu\bigl(\{s \in [\latch{i}{j}{\pricecentral-\priceleft},+\infty) \colon  d(s-t_j) + \pricecentral\leq 0 \}\bigl)
$. 

# \label{A4} 
{\em 
When $a = \bigl((0,\cdot),(j,\pricecentral)\bigl)\in A_1$ and $a' =  \bigl((j,\pricecentral),(n+1,\cdot)\bigl)\in A_3$.} 
Then set
\[
w(a,a') \coloneqq \left\{\begin{array}{ll}\pricecentral v(a,a') & \text{if $v(a,a')\leq \nu_j$,} \\ -\infty & \text{otherwise,}
\end{array}\right.
\]
where
$ 
v(a,a') \coloneqq \mu\bigl(\{s\in\Real\colon  d(s-t_j) + \pricecentral\leq  0\}\bigl)
$. 

\end{easylist}

We consider this reward to be attached to the corresponding arc in the line digraph of~$D$, denoted by~$L(D)$. (We remind the reader that the \emph{line digraph} has~$A$ as its vertex set and an arc~$(a,a')$ for each pair of consecutive arcs~$a,a'$ in~$D$.) As usual, the reward~$w(C)$ of a path~$C$ in~$L(D)$ is defined as the sum of the rewards on its arcs. We are thus interested in paths in~$L(D)$ whose induced paths in~$D$ go from the source to the sink.
To lighten the text, we call source-to-sink  path in~$L(D)$ the line path of any source-to-sink path in~$D$. 
 
The idea of the construction is to model the price profiles as source-to-sink paths in $D$. Roughly speaking, the vertices visited by such a path indicate which time slots 
should be considered for selection by the users,
and which prices may be attached to this time slots. The conditions in the definition of $A_1$, $A_2$, and $A_3$ ensure that it is always possible to choose prices for unused time slots that keep them unattractive to all users.
The revenue generated by a vertex of such a path depends on the vertices just before and just after on the path, and this is the reason why it is more convenient to consider~$L(D)$.
See \cref{fig:path} for an illustration of a source-to-sink path in~$D$ and its corresponding line path in~$L(D)$.
%

\begin{figure}[ht]
\begin{center}
\includegraphics[]{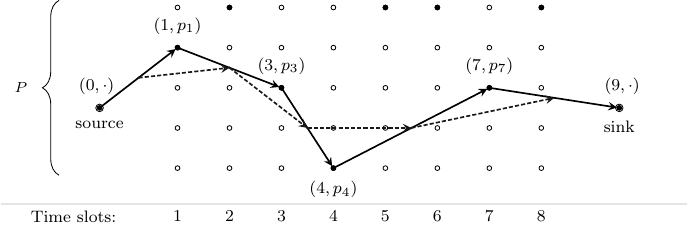}
\end{center}
\caption{\label{fig:path}
Example of a source-to-sink path 
in~$D$ with its associated path in~$L(D)$ (dashed lines) when $n=8$ and~$P$ has five elements. 
The corresponding price profile derived in \cref{lem:path-price} is given by the sequence of black dots.
}
\end{figure}

For the next statements, we denote by $R(p)$ the revenue $\sum_{j=1}^n\mu(I_j(p))$ when~$p$ is a feasible solution of~\eqref{upper-level}. We set $R(p)\coloneqq -\infty$ when~$p$ is not a feasible solution of~\eqref{upper-level}.

\begin{lemma}\label{lem:price-path}
Given a price profile $p \in P^n$, there exists a source-to-sink path~$C$ in~$L(D)$ with $R(p) \leq w(C)$.
\end{lemma}

\begin{lemma}\label{lem:path-price}
Given a source-to-sink path~$C$ in~$L(D)$, there exists a price profile~$p \in P^n$ with $R(p) \geq w(C)$. Such a profile can be built in linear time.
\end{lemma}

\Cref{lem:price-path,lem:path-price} show that computing an optimal price profile boils down to computing a source-to-sink path in~$L(D)$ of maximal reward. Since~$D$ is acyclic, $L(D)$ is acyclic as well, and such a computation can be done in linear time in the number of arcs of~$L(D)$ (see, e.g., \cite{moravek70},\cite[Theorem 8.16]{schrijver03}), i.e., in $O(|A||V|) = O(n^3|P|^3)$.
The construction of~$D$ itself can be done in the same time complexity.
This suffices to prove  Theorem~\ref{thm:semi-disc}.
Observe that the price profile of \cref{fig:intervals} and the source-to-sink path of  \cref{fig:path} are related through \cref{lem:price-path,lem:path-price}.

\subsection{Proofs of the Lemmas}

\begin{proof}[
\Cref{lem:price-path}]
Let $p=(p_1,\ldots,p_n) \in P^n$. If $\mu(I_j(p)) > \nu_j$ for at least one $j$, then we have $R(p) = -\infty$, and in particular $R(p) \leq w(C)$ for any source-to-sink path~$C$. So we assume that $\mu(I_j(p)) \leq \nu_j$ for all $j$. Consider the~$m$ time slots $
j_1, \dots , j_m
$ such that $
j_1 < j_2 < \cdots < j_m
$ and 
$\fullIntervali{j_i}{p}$ is an interval of non-zero length for every $i$, where $1\leq m \leq n$. The sequence of vertices 
$(0,\cdot),(j_1,p_{j_1}),\dots,(j_m,p_{j_m}),(n+1,\cdot)$ 
induces a source-to-sink path in $D$ by \cref{prop:intervals}: 
the endpoints of the intervals $\fullIntervali{j_1}{p},\dots,\fullIntervali{j_m}{p}$ are precisely the points $-\infty,\latch{j_1}{j_2}{p_{j_2}-p_{j_1}},\dots,\latch{j_{m-1}}{j_{m}}{p_{j_{m}}-p_{j_{m-1}}},+\infty$.
Set $(j_0,p_0)=(0,\cdot)$ and $(j_{m+1},p_{n+1})=(n+1,\cdot)$, 
and consider two consecutive arcs $a_i=\big((j_{i-1},p_{j_{i-1}}),(j_{i},p_{j_{i}})\big)$ and $a_{i+1}=\bigl((j_{i},p_{j_i}),(j_{i+1},p_{j_{i+1}})\bigl)$ of this path. By construction, $v(a_i,a_{i+1}) = \mu(\Intervali{j_{i}}{p})$ is the service load at~$j_{i}$ and $w(a_i,a_{i+1}) = p_{j_{i}}\mu(\Intervali{j_{i}}{p})$. The path induces in turn a source-to-sink path $C$ in $L(D)$ such that $R(p) = w(C)$.
\qed\end{proof}

\begin{proof}[
\Cref{lem:path-price}]
\newcommand{\realizedset}{U}%
Let~$C$ be a source-to-sink path in $L(D)$ with $m$ arcs denoted by $(a_0,a_1),\dots,$ $(a_{m-1},a_{m})$, where $1\leq m \leq n$. 
Let $
(j_1,p_{j_1}),\dots,$ $(j_m,p_{j_m})
$ be the internal vertices of the path induced by $C$ in $D$,
and set $\realizedset=\{j_1,\dots,j_m\}$. For each $j\notin\realizedset$, 
set $p_j \coloneqq \max P$. This defines a price profile $p=(p_1,\ldots,p_n)$. If $w(a,a')=-\infty$ for an arc $(a,a')$ of $C$, then we have $w(C)=-\infty$, and in particular $R(p) \geq w(C)$.
So we assume $w(a,a') > -\infty$ for each arc $(a,a')$ of $C$. 

In the rest of the proof, we consider the sequence 
$s_1,\dots,s_{m-1}$, where
$s_{i} \coloneqq \latch{j_{i}}{j_{i+1}}{p_{j_{i+1}}-p_{j_{i}}}$. When $m>1$, this sequence is nonempty, and for $i\in[m-1]$ we have $\big((j_{i},p_{j_{i}}),((j_{i+1},p_{j_{i+1}}))\big)\in A_2$ and $w\big((j_{i},p_{j_{i}}),((j_{i+1},p_{j_{i+1}}))\big)>-\infty$, which imply $-\infty < s_1 
\leq \cdots \leq s_{m-1} < +\infty$.

First assume that $m>1$. We claim that 
$-\infty , s_1 
, \dots , s_{m-1} , +\infty$
are the endpoints of $\fullIntervali{j_1}{p},\dots,\fullIntervali{j_m}{p}$.
To prove this, it is enough to show that
the inequality
\begin{equation}\label{theinequality}
\dist{\xvar-\yvari{j_{i}}}+\pricei{j_{i}} \leq \dist{\xvar-\yvari{\ell}}+\pricei{\ell}
\end{equation} 
holds for $j_{i}\leq \ell \leq n$ when $i \in [m-1]$ and $\xvar\leq s_{i}$, and when $i=m$ and $\xvar\in\Real$; and the  inequality
$\dist{\xvar-\yvari{j_{i}}}+\pricei{j_{i}} \leq \dist{\xvar-\yvari{\ell}}+\pricei{\ell}$ holds for $
1 \leq \ell \leq j_{i}$ when $i=1$ and $\xvar\in\Real$, and when $i-1 \in [m-1]$ and $\xvar\geq s_{i-1}$. 
Indeed, combining these inequalities yields $\fullIntervali{j_1}{p}\supseteq (-\infty,s_{1}] , \fullIntervali{j_2}{p}\supseteq [s_{1},s_{2}] ,\dots, \fullIntervali{j_{m-1}}{p}\supseteq [s_{m-2},s_{m-1}], \fullIntervali{j_m}{p} \supseteq [s_{m-1},+\infty) $.
The claim then follows from \cref{prop:intervals}.
Suppose now that $m=1$. We claim that~$\fullIntervali{j_1}{p}=\Real$. To prove this, it is enough to show that the same two inequalities hold for $i=1=m$ and~$\xvar\in\Real$. 
To establish the two claims, 
we only verify the inequality~\eqref{theinequality}; the second one can be checked similarly.
We do this by considering~$i\in [m]$, $j_{i}\leq \ell \leq n$, and~$\xvar\in\Real$ in the following three cases:

\smallskip
\begin{easylist}
\ListProperties(Numbers=r,Mark=.,FinalMark={)},FinalMark3={)})

\item $\bullet$
\emph{When $i=m$ and~$\xvar\in\Real$.}
Then, we have $\ell\geq j_{m}$, which implies $\ell\notin \realizedset$ and  $p_{\ell} = \max P$.
Since~$C$ is a source-to-sink path in $L(D)$, we infer from the definition of the arcs of~$D$ that
~$\ell$ is covered by $(j_m,p_{j_m})$ at $+\infty$.
It follows from \cref{fact:increasing} and $j_m \leq \ell$ that
$
\dist{\xvar-\yvari{j_{m}}}
-
\dist{\xvar-\yvari{\ell}}
\leq
\dist{\altxvar-\yvari{\yvari{j_{m}}}}
-
\dist{\altxvar-\yvari{\ell}}
$ holds for all $\altxvar\geq\xvar$.
Equivalently, we have $
\dist{\xvar-\yvari{j_{m}}}+\pricei{j_{m}} 
\leq 
\dist{\xvar-\yvari{\ell}}
+
\pricei{j_{m}} 
+
\dist{\altxvar-\yvari{\yvari{j_{m}}}}
-
\dist{\altxvar-\yvari{\ell}}
$ for all $\altxvar\geq\xvar$.
Using the fact that~$\ell$ is covered and letting $\altxvar\to+\infty$ yields 
$
\dist{\xvar-\yvari{j_{m}}}+\pricei{j_{m}} 
\leq
\dist{\xvar-\yvari{\ell}}
+
\pricei{j_{m}} 
+
\max P
-
\pricei{j_{m}} 
=
\dist{\xvar-\yvari{\ell}}
+
\pricei{\ell}
$.
Hence, the inequality~\eqref{theinequality} holds.

\item $\bullet$
\emph{When~$1\leq i<m$, $\ell\in\realizedset$, and $\xvar\leq s_{i}$.} Then, we show~\eqref{theinequality} by induction on~$k$.
For~$k=i$, the inequality~\eqref{theinequality} is immediate. Now, suppose it holds for $k\in \{i,\dots,m -1\}$.
Using successively \cref{fact:increasing} with $\xvar\leq s_{i}\leq s_{k}$ and the definition of~$\latch{j_{k}}{j_{k+1}}{p_{k+1}-p_{k}}$, we find $\dist{\xvar-\yvari{j_{k}}}
-
\dist{\xvar-\yvari{j_{k+1}}}
\leq
\dist{s_{k}-\yvari{j_{k}}}
-
\dist{s_{k}-\yvari{j_{k+1}}}
=
\pricei{j_{k+1}}-\pricei{j_{k}}
$.
Then, 
combining the induction hypothesis with the last inequality,
we find
$
\dist{\xvar-\yvari{j_{i}}}+\pricei{j_{i}} 
\leq
\dist{\xvar-\yvari{j_{k}}}+\pricei{j_{k}} 
\leq
\dist{\xvar-\yvari{j_{k+1}}}
+
\pricei{j_{k+1}} 
$. 
Hence, the inequality~\eqref{theinequality} holds.

\item $\bullet$
\emph{When~$1\leq i<m$, $\ell\notin\realizedset$, and $\xvar\leq s_{i}$.} Then,
we have $p_{\ell} = \max P$. 
Since~$C$ is a source-to-sink path in $L(D)$, we infer from the definition of the arcs of~$D$ that there is $k\in\{i,\dots,m\}$ such that $j_k\leq \ell$ and~$\ell$ is covered either by $(j_k,p_{j_k})$ at $s_{k}$ when $i\leq k < m$, or by $(j_m,p_{j_m})$ at $+\infty$ when $k=m$.
In the first case, 
it follows from \cref{fact:increasing}, $j_k \leq \ell$, and $\xvar\leq s_{k}$ that
$
\dist{\xvar-\yvari{j_{k}}}
-
\dist{\xvar-\yvari{\ell}}
\leq
\dist{s_{k}-\yvari{j_{k}}}
-
\dist{s_{k}-\yvari{\ell}}
$.
Applying successively the inequality~\eqref{theinequality} at~$j_k$, the fact that~$\ell$ is covered, and
the last inequality,
we find 
$
\dist{\xvar-\yvari{j_{i}}}+\pricei{j_{i}} 
\leq
\dist{\xvar-\yvari{j_{k}}}+\pricei{j_{k}} 
\leq 
\dist{s_{k}-\yvari{\ell}}
+
\max P
+
\dist{\xvar-\yvari{j_{k}}}
-
\dist{s_{k}-\yvari{j_{k}}}
\leq
\dist{\xvar-\yvari{\ell}}
+
\pricei{\ell} 
$.
In the second case, 
applying the inequality~\eqref{theinequality} successively at~$\xvar\leq s_{i}$ for~$j_i\leq j_m$ and then at~$\xvar\in\Real$ for $  j_m\leq \ell$, we find $
\dist{\xvar-\yvari{j_{i}}}+\pricei{j_{i}} 
\leq
\dist{\xvar-\yvari{j_{m}}}+\pricei{j_{m}} 
\leq 
\dist{\xvar-\yvari{\ell}}
+
\pricei{\ell}
$.
Hence, the inequality~\eqref{theinequality} holds. 

\end{easylist}

\smallskip

Thus, we have shown 
\[
\begin{array}{r}
\fullIntervali{j_1}{p}=(-\infty,s_{1}] \, 
,\ 
\fullIntervali{j_2}{p}=[s_{1},s_{2}] \, 
,\ \dots,\
\fullIntervali{j_{m-1}}{p}=[s_{m-2},s_{m-1}]\, 
,\ 
\qquad\\
\fullIntervali{j_m}{p}=[s_{m-1},+\infty)\, .
\end{array}
\]
This implies $v(a_{i-1},a_{i}) = \xmeasure{\{\xvar\in\fullIntervali{j_i}{p}\colon \dist{\xvar-\yvari{j_i}} +  \pricei{j_i}\leq 0\}} = \xmeasure{\Intervali{j_i}{p}}$ for  $i=1,\ldots,m$, and therefore $\xmeasure{\Intervali{j_i}{p}} \leq \ymeasurei{j_i}$ and $\pricei{j_i}\xmeasure{\Intervali{j_i}{p}} = w(a_{i-1},a_{i})$. 
Since $\fullIntervali{j_1}{p},\dots,\fullIntervali{j_m}{p}$ cover~$\Real$, for every~$\ell$ that is not realized as $j_i$ we have 
$\xmeasure{\Intervali{\ell}{p}} = 0$. 
Hence $R(p) = \sum_{\ell=1}^{n} \pricei{\ell}\xmeasure{\Intervali{\ell}{p}} = w(C)$.
\qed\end{proof}


\section{Algorithm when \texorpdfstring{$P$}{\textit{P}} Is \texorpdfstring{$\Real$}{\textbf{R}}} \label{sec:algoPisR}

This section is devoted to the proof of~\cref{thm:cont}, which deals with the case $P=\Real$. 
Throughout \cref{sec:algoPisR} we assume that~$\distfunction$ is strongly convex with parameter~$\stronglyconvex$ and that the density of~$\xmeasurefunction$  is interval supported and  lower- and upper-bounded above~$0$.
We respectively denote the corresponding bounds on the density by~$\xmeasuremin>0$ and~$\xmeasuremax$.
We also define the constant
$\lipschitz= 2 \max\big\{[\stronglyconvex (\yvari{i+1}-\yvari{i})]^{-1} \, \colon \, i\in[n-1] \big\}$, which plays a role in our analysis, as well as the price thresholds
$\pricemax \coloneqq -d(0)$  and
\[
\pricemin \coloneqq \min\left(\pricemax,
\inf
\big\lbrace d(s-t_j)-d(s-t_k)
\colon s \in \supp(\xmeasurefunction)
, j,k\in[n] 
\big\rbrace\right)
\, .
\]
Lastly, we adopt the following notation: for any~$q\in\Real$, we write  $q^{\delta}=\lfloor q /\delta \rfloor\delta$; 
for any~$\profile\in\Real^n$, we write  $\profilemdelta{\delta}=(\pricemdeltai{\delta}{1},\dots,\pricemdeltai{\delta}{n})$.

\subsection{Lower Bound}\label{sec:lowerbound}

As a lower bound on~$\OPT{\Real}$, we consider $\operatorname{LB}(\delta)=\OPT{\delta\Integer}$,  where $\delta>0$. The computation of~$\operatorname{LB}(\delta)$ rests on the following result.

\begin{proposition}\label{prop:bounded}
We have $\OPT{\Real} = \OPT{[\pricemin,\pricemax]}$. Moreover, for every $\delta>0$, we have $\OPT{\delta\Integer} = \OPT{[\pricemin,\pricemax+\delta} \cap \delta\Integer)$.
\end{proposition}
%
%
\begin{proof}

We first deal with the case $P=\Real$. Setting $p_k$ to $p^{\max}$ for every $k$ yields a feasible solution giving $0$ to the objective function, which implies $\OPT{\Real} \geq 0$. If $\OPT{\Real} = 0$, then the solution with all $p_k$ equal to $p^{\max}$ is optimal and we have $\OPT{\Real} = \OPT{[\pricemin,\pricemax]}$. We can thus suppose $\OPT{\Real} > 0$ and consider a feasible solution $p \in \Real^n$ giving a positive value to the objective function. Pick a $j \in [n]$ such that $\mu(I_j(p)) > 0$ and $p_j >0$.
By the definition of $j$, there is a point $s \in \supp(\xmeasurefunction)$ such that $d(s-t_j) + p_j \leq d(s-t_k) + p_k $ for every $k \in [n]$, yielding $   p_k  \geq d(s-t_j) - d(s-t_k) + p_j \geq \pricemin +  p_j > \pricemin$. Thus, $p_k > p^{\min}$ for every $k \in [n]$. 
Since setting $p'_k\coloneqq \min(p_k, p^{\max})$ provides a feasible solution $p'$ that does not decrease the objective function, we get the equality $\OPT{\Real} = \OPT{[\pricemin,\pricemax]}$.

We then deal with the case $P=\delta\Integer$. Similarly as for the previous case, setting $p_k$ to $q \coloneqq \lceil p^{\max} / \delta\rceil\delta$ yields a feasible solution giving $0$ to the objective function, which implies $\OPT{\delta\Integer} \geq 0$. If $\OPT{\delta\Integer}=0$, then $\OPT{\delta\Integer} = \OPT{[\pricemin,\pricemax+\delta} \cap \delta\Integer)$ for the same reasons as above. We can thus suppose $\OPT{\delta\Integer}>0$ and consider a feasible solution $p \in (\delta\Integer)^n$ giving a positive value to the objective function. 
Again, as above, we necessarily have $p_k > p^{\min}$ for every $k \in [n]$. Since setting $p'_k\coloneqq \min(p_k, q)$ provides a
feasible solution~$p'$ that does not decrease the objective function, we get the equality $\OPT{\delta\Integer} = \OPT{[\pricemin,\pricemax+\delta} \cap \delta\Integer)$.
\qed\end{proof}

By continuity of the objective and the constraints in~\eqref{upper-level} and by compactness, an immediate consequence of this proposition is that there always exists an optimal solution to problem~\eqref{upper-level} on~$[\pricemin,\pricemax]$ when $P = \Real$, and one on~$[\pricemin,\pricemax+\delta ) \cap \delta\Integer$ when $P = \delta\Integer$.
Combining \cref{prop:bounded} with \cref{thm:semi-disc} implies that the computation of the lower bound $\operatorname{LB}(\delta)
$ can be done in time $O(\frac 1 {\delta^3} n^3)$. 

\subsection{Upper Bound and Proof of~\cref{thm:cont}}\label{sec:upperbound}
Given~$\delta>0$, 
consider the following 
``approximate'' 
version of~\eqref{upper-level}:
\begin{equation}\label{upper-level-alter-delta}\tag{U${}^\delta$}
\begin{array}{r@{\qquad}l@{\quad}l}
\underset{\hspace{-10mm} \pricei{1},\ldots,\pricei{n}\hspace{-10mm}}{\text{maximize}} & \ds{\sum_{j=1}^n 
(\pricemdeltai{\delta}{j} + \delta )
\xmeasure{\Intervali{j}{\profilemdelta{\delta}}}
}
\\
\text{subject to} 
 &
 \xmeasure{\Intervali{j}{\profilemdelta{\delta}}}
\leq
\ymeasurei{j} + \xmeasuremax\lipschitz \delta 
+ 2\xmeasuremax\sqrt{2 \delta /\stronglyconvex}  
& \forall j \in [n]
 \\ & \pricei{j} \in \Real
 & \forall j \in [n] \, 
 .
\end{array}
\end{equation}
\noindent 
Denote by $\OPTdelta{\delta}{\Real}$ its optimal value.

We first state a counterpart to~\cref{prop:bounded} for the modified problem~\eqref{upper-level-alter-delta}.
\begin{proposition}\label{prop:bounded-delta}\label{lem:upper}
For every $ \delta>0$, we have $\OPTdelta{\delta}{\Real}
+n\xmeasuremax\lipschitz \delta\pricemax
\geq \OPT{\Real}$. Moreover, \eqref{upper-level-alter-delta} admits an optimal solution in~$\big([\pricemindelta{\delta}-\delta,\pricemax+\delta]\cap\delta\Integer\big)^n$.
\end{proposition}

\noindent

Our proof for ~\cref{thm:cont} then rests on the following result.

\begin{proposition}\label{prop:upper_bound}
We have
 $\OPTdelta{\delta}{\Real} +n\xmeasuremax\lipschitz \delta\pricemax
 - \OPT{\delta\Integer} =O(\delta^{1/8})$.
\end{proposition}

With this proposition, and the algorithm of Section~\ref{sec:algoPfinite}, the proof of~\cref{thm:cont} is somehow routine.

\begin{proof}[
\cref{thm:cont}]
\Cref{lem:upper} ensures that $\operatorname{UB}(\delta)=\OPTdelta{\delta}{\Real}+n\xmeasuremax\lipschitz \delta\pricemax$ is an upper bound on $\OPT{\Real}$. On the other hand,~\cref{prop:upper_bound} ensures that $
\OPTdelta{\delta}{\Real}+n\xmeasuremax\lipschitz \delta\pricemax - \OPT{\delta\Integer} = O(\delta^{1/8})$. Finally, the algorithm of Section~\ref{subsec:algoPfinite} can be adapted in a straightforward way to solve~\eqref{upper-level-alter-delta} in~$O(n^3|P|^3)$.
\qed\end{proof}

\subsection{Proofs of \cref{prop:upper_bound,prop:bounded-delta}} \quad

\subsubsection{Preliminaries}

We start by establishing two basic lemmas about strongly convex functions.
Recall that a function~$f$ is strongly convex with parameter~$\alpha$ if
\begin{equation}\label{strongconvexity}
f(\theta x + (1-\theta) y)
\leq
\theta f(x) + (1-\theta) f(y) 
-
\frac{\alpha}{2} \theta(1-\theta) (x - y)^2
\end{equation}
holds for all $x,y \in \Real$ and $\theta\in[0,1]$; see~\cite{nesterov14}.

\begin{lemma}\label{lem:strong-lb}
    Let $f\colon \Real \to \Real$ be a strongly convex function with parameter~$\alpha$ and reaching its global minimum at $0$. Then $f(y) - f(x) \geq \frac 1 2 \alpha (y-x)^2$ for all $0 \leq x \leq y$.
\end{lemma}

\begin{proof}
For $\theta \in [0,1)$, we have by~\eqref{strongconvexity}
\[     f(y) - f(x) \geq \frac 1 {1-\theta} \left(f\bigl(\theta x+(1-\theta)y\bigl) - f(x) \right) + \frac \alpha 2 \theta  (x-y)^2 
\geq \frac \alpha 2 \theta  (y-x)^2
\, ,       \]
where we used $\theta x+(1-\theta)y\geq x$ and the fact that~$f$ does not decrease on nonnegative numbers. We get the desired inequality by making~$\theta$ go to $1$.%
\qed\end{proof}

\begin{lemma}\label{lem:rate-strong}
    Let $f\colon \Real \to \Real$ be a strongly convex function with parameter $\alpha$. Then for all $x \leq y $ and $h >0$, we have $f(y+h)- f(y) \geq f(x+h) - f(x) + \alpha h (y-x)$.
\end{lemma}

\begin{proof}
Set $\theta = h/(y+h-x)$, and observe that $x+h = \theta (y+h) + (1-\theta) x$. Applying~\eqref{strongconvexity} at~$x+h$, we find $ f(x+h) \leq \theta f(y+h) + (1-\theta) f(x) - \frac{\alpha}{2} h (y-x) $, where we have used $\theta(1-\theta)(y+h-x)^2 = h (y-x)$.
Similarly, $ y = \theta x + (1-\theta) (y+h) $ yields $ f(y) \leq \theta f(x) + (1-\theta) f(y+h) - \frac{\alpha}{2} h (y-x) $. The desired inequality follows by adding the last two results and rearranging the terms.  \qed\end{proof}

Our proofs for \cref{prop:bounded-delta,prop:upper_bound} rest on a series of ancillary results.
We emphasize that the sets considered in the sequel will be measurable thanks to \cref{prop:intervals}.
In the upcoming sections, we use~$|X|$ to denote the Lebesgue measure of a measurable set~$X$, and~$A\symdif B=(A\setminus B)\cup(B\setminus A)$ to denote the symmetric difference between two sets~$A$ and~$B$.

\subsubsection{Proof of \cref{prop:bounded-delta}}

The next result relies on \cref{lem:rate-strong}.
We recall that $\Intervali{j}{\profilemdelta{\delta}} 
=\fullIntervali{j}{\profilemdelta{\delta}} \cap
\sublevelseti{j}{\pricemdeltai{\delta}{j}} 
$, where~$\fullIntervalifunction{j}
$ and~$\sublevelsetifunction{j}$ are defined in \cref{sec:preliminaries}.

\begin{lemma}\label{lemma:twolemmasinone}
\renewcommand{\altprofile}{\altpricefunction}%
\renewcommand{\altpricefunction}{p'}%
For all $\eta,\eta'\geq0$, all~$\profile,\altprofile\in\Real^n$ such that $\pricei{k}-\eta'\leq\altpricei{k}\leq\pricei{k}+\eta$ for $k\in[n]$, and all~$j\in[n]$, we have $
|\fullIntervali{j}{\altprofile}\setminus\fullIntervali{j}{\profile}|
\leq
(\eta+\eta')\lipschitz
$ and $
|\fullIntervali{j}{\profile}\setminus \fullIntervali{j}{\altprofile}| \leq (\eta+\eta')\lipschitz 
$. 
\end{lemma}
\begin{proof}
\renewcommand{\altprofile}{\altpricefunction}%
\renewcommand{\altpricefunction}{p'}%
\newcommand{\pricedistance}{\tau}%

Since $\pricei{k}-\eta'\leq\altpricei{k}\leq\pricei{k}+\eta$ rewrites as $\altpricei{k}-\eta\leq\pricei{k}\leq\altpricei{k}+\eta'$, we only show the first inequality. The second inequality can be shown similarly by exchanging the roles of~$\profile$ and~$\altprofile$.

We focus on $\fullIntervali{j}{\altprofile}\setminus\fullIntervali{j}{\profile}$, which we rewrite as~$E\cup F$, where
\[
\begin{array}{l}  
E = \{
\xvar\in\fullIntervali{j}{\altprofile} \colon \exists i< j, \, 
\dist{\xvar-\yvari{i}}+\pricei{i}
< 
\dist{\xvar-\yvari{j}}+\pricei{j}
\}
\, ,
\\
F = \{
\xvar\in\fullIntervali{j}{\altprofile} \colon \exists k > j , \, 
\dist{\xvar-\yvari{k}}+\pricei{k}
< 
\dist{\xvar-\yvari{j}}+\pricei{j}
\}
\, .
\end{array}
\]
%
Let~$\xvar\in E$ and~$\xivariable\geq\eta+\eta'$. 
By successively using the assumption on $\profile$ and~$\altprofile$, the definition of~$E$, and \cref{lem:rate-strong} (with $x=\xvar-\frac{1}{2}\lipschitz\xivariable-\yvari{j}$, $y=\xvar-\yvari{j}$, and $h=\yvari{j}-\yvari{i}$), we find $i<j$ such that
$
\altpricei{j}-\altpricei{i}
\geq
\pricei{j}-\pricei{i} -\eta-\eta'
>
\dist{\xvar-\yvari{i}} - \dist{\xvar-\yvari{j}} -\eta-\eta'
\geq
\dist{\xvar-\frac{1}{2}\lipschitz\xivariable-\yvari{i}} - \dist{\xvar-\frac{1}{2}\lipschitz\xivariable-\yvari{j}} + \frac{1}{2}\lipschitz\stronglyconvex(\yvari{j}-\yvari{i})\xivariable -\eta-\eta'
\geq
\dist{\xvar-\frac{1}{2}\lipschitz\xivariable-\yvari{i}} - \dist{\xvar-\frac{1}{2}\lipschitz\xivariable-\yvari{j}} 
$, which implies $\xvar-\frac{1}{2}\lipschitz\xivariable\notin\fullIntervali{j}{\altprofile}$ and therefore $\xvar-\frac{1}{2}\lipschitz\xivariable\notin E$. Since this holds for all~$\xivariable\geq\eta+\eta'$,
we have 
$|E|\leq \frac{1}{2}\lipschitz(\eta+\eta')$.
Similarly, let~$\xvar\in F$ and~$\xivariable\geq\eta+\eta'$. 
By successively using \cref{lem:rate-strong} (with $x=\xvar-\yvari{k}$, $y=\xvar+\frac{1}{2}\lipschitz\xivariable-\yvari{k}$, and $h=\yvari{k}-\yvari{j}$), the definition of~$F$, and the assumption on $\profile$ and~$\altprofile$, we find $k>j$ such that
$
\dist{\xvar+\frac{1}{2}\lipschitz\xivariable-\yvari{j}} - \dist{\xvar+\frac{1}{2}\lipschitz\xivariable-\yvari{k}} 
\geq
\dist{\xvar-\yvari{j}} - \dist{\xvar-\yvari{k}} +\frac{1}{2}\lipschitz\stronglyconvex(\yvari{k}-\yvari{j})\xivariable
>
\pricei{k}-\pricei{j}
+\eta+\eta' 
\geq
\altpricei{k}-\altpricei{j} 
$, which implies $\xvar+\frac{1}{2}\lipschitz\xivariable\notin\fullIntervali{j}{\altprofile}$ and therefore $\xvar+\frac{1}{2}\lipschitz\xivariable\notin F$. 
Since this holds for all~$\xivariable\geq\eta+\eta'$,
we have 
$
|F|\leq \frac{1}{2}\lipschitz(\eta+\eta')$.
Thus, $
|\fullIntervali{j}{\altprofile}\setminus\fullIntervali{j}{\profile}|
\leq
|E|+|F|
\leq
\lipschitz(\eta+\eta')
$. 
\qed\end{proof}

Next, we show that the length of the sublevel set~$\sublevelsetifunction{j}$ is Hölder continuous with  exponent~${1}/{2}$.

\begin{lemma}\label{lemma:sublevelsets}
For all $\altaltprice,\altaltprice^\prime\in\Real$ and all~$j\in[n]$,
 we have
$
|\sublevelseti{j}{\altaltprice}\symdif \sublevelseti{j}{\altaltprice^\prime}| 
 \leq
2\sqrt{ {2}| \altaltprice^\prime-\altaltprice
 |/{\stronglyconvex} }  
$.
\end{lemma}
\begin{proof}
Without loss of generality, we assume $\altaltprice\leq\altaltprice^\prime$.  
We have
$\sublevelseti{j}{\altaltprice}\supseteq\sublevelseti{j}{\altaltprice^\prime}$ and  it follows from \cref{prop:intervals} that
$|\sublevelseti{j}{\altaltprice}\symdif \sublevelseti{j}{\altaltprice^\prime}| = |\sublevelseti{j}{\altaltprice}\setminus \sublevelseti{j}{\altaltprice^\prime}|= |\sublevelseti{j}{\altaltprice}| -|\sublevelseti{j}{\altaltprice^\prime}|$.
It remains to show that $
 |\sublevelseti{j}{\altaltprice}| -|\sublevelseti{j}{\altaltprice^\prime}| 
 \leq
2\sqrt{ {2}( \altaltprice^\prime-\altaltprice
 )/{\stronglyconvex} } $.

We first assume that $\altaltprice\leq\altaltprice^\prime\leq -\min\distfunction$, so that~$-\altaltprice$ and~$-\altaltprice^\prime$ belong to the codomain of~$\distfunction$.
Since~$\xvar\mapsto\dist{\xvar-\yvari{j}}+ \altaltprice$ is strongly convex, its sublevel sets are compact sets and therefore finite intervals of~$\Real$. It follows that we can find $\zvar\leq\altzvar\leq\yvari{j}\leq\altxvar\leq\xvar$
such that 
$\dist{\zvar-\yvari{j}}=\dist{\xvar-\yvari{j}}=- \altaltprice$ and $\dist{\zvar^\prime-\yvari{j}}=\dist{\xvar^\prime-\yvari{j}}=- \altaltprice^\prime$. 
Using \cref{lem:strong-lb} with $f=\distfunction$, $x=\altxvar-\yvari{j}$, and $y=\xvar-\yvari{j}$, we find $
\altaltprice^\prime-\altaltprice
=
\dist{\xvar-\yvari{j}} - \dist{\altxvar-\yvari{j}}
\geq \frac{1}{2} \stronglyconvex (\xvar-\altxvar)^2
$, which implies $\xvar-\altxvar\leq \sqrt{2(\altaltprice^\prime-\altaltprice)/\stronglyconvex}$.
Using \cref{lem:strong-lb} with $f\colon\xvar\mapsto\dist{-\xvar}$, $x=\yvari{j}-\altzvar$, and $y=\yvari{j}-\zvar$, we find $\altaltprice^\prime-\altaltprice = \dist{\zvar-\yvari{j}} - \dist{\altzvar-\yvari{j}}\geq \frac{1}{2} \stronglyconvex (\altzvar-\zvar)^2
$, which implies $\altzvar-\zvar\leq \sqrt{2(\altaltprice^\prime-\altaltprice)/\stronglyconvex}$.
Combining the last two results gives 
$
 |\sublevelseti{j}{\altaltprice}| -|\sublevelseti{j}{\altaltprice^\prime}| 
=
 (\xvar - \zvar)-(\altxvar - \altzvar)
 \leq
2\sqrt{ {2}( \altaltprice^\prime-\altaltprice
 )/{\stronglyconvex} } 
 $.

\newcommand{\generalprice}{\tilde{\altaltprice}}
In the general case, consider $\generalprice=\min(\altaltprice,-\min\distfunction)$ and $\generalprice^\prime=\min(\altaltprice^\prime,-\min\distfunction)$. Note that $\generalprice
\leq
\generalprice^\prime$.
By construction, we have $|\sublevelseti{j}{\altaltprice}|=|\sublevelseti{j}{\generalprice}|$ and $ |\sublevelseti{j}{\altaltprice^\prime}|=|\sublevelseti{j}{\generalprice^\prime}|$.
Since the projection $x\mapsto\min(x,-\min\distfunction)$ is non-expansive, we find $\generalprice^\prime-\generalprice\leq\altaltprice^\prime-\altaltprice$. Applying the previous result to $\generalprice\leq\generalprice^\prime\leq-\min\distfunction$ yields $
 |\sublevelseti{j}{\altaltprice}| -|\sublevelseti{j}{\altaltprice^\prime}| 
=
|\sublevelseti{j}{\generalprice}| -|\sublevelseti{j}{\generalprice^\prime}| 
 \leq
2\sqrt{ {2}( \generalprice^\prime-\generalprice
 )/{\stronglyconvex} } 
 \leq
2\sqrt{ {2}( \altaltprice^\prime-\altaltprice
 )/{\stronglyconvex} } 
 $.
\qed\end{proof}

The next lemma is a consequence of \cref{lemma:twolemmasinone,lemma:sublevelsets}. 
It tells us to what extent a feasible price profile of the modified problem~\eqref{upper-level-alter-delta} may violate the capacity constraints in the original problem~\eqref{upper-level}.

\begin{lemma}
\label{lemma:doubleinequalities-part-two:new}
Let~$\delta>0$.
For all~$\profile\in\Real^n$ and~$j\in[n]$, we have
 \begin{equation*} 
\xMeasure{\Intervali{j}{\profile}} 
- \xmeasuremax\lipschitz \delta
\leq
\xMeasure{\Intervali{j}{\profilemdelta{\delta}}}
\leq
\xMeasure{\Intervali{j}{\profile}} + \xmeasuremax\lipschitz \delta 
+ 2\xmeasuremax\sqrt{2 \delta /\stronglyconvex}  
\, . %
\end{equation*}
\end{lemma}
\begin{proof}
\renewcommand{\altprofile}{\altpricefunction}%
\renewcommand{\altpricefunction}{p'}%

Applying \cref{lemma:twolemmasinone} to~$\profile$ and~$
\profilemdelta{\delta}$  with parameters $\eta=0$ and $\eta'=\delta$, we find
$
|\fullIntervali{j}{\profile}\setminus \fullIntervali{j}{\profilemdelta{\delta}}| \leq \lipschitz \delta
$ and $
|\fullIntervali{j}{\profilemdelta{\delta}}\setminus\fullIntervali{j}{\profile}|
\leq
\lipschitz \delta
$, which imply
$
\xMeasure{\fullIntervali{j}{\profile}\setminus \fullIntervali{j}{\profilemdelta{\delta}}} \leq \xmeasuremax\lipschitz \delta
$ and $
\xMeasure{\fullIntervali{j}{\profilemdelta{\delta}}\setminus\fullIntervali{j}{\profile}}
\leq
\xmeasuremax\lipschitz \delta
$.

Now, it
%
%
follows from~$\pricemdeltai{\delta}{j}\leq \pricei{j}$ and the definition of~$\sublevelsetifunction{j}$ that $\sublevelseti{j}{\pricei{j}}\subseteq \sublevelseti{j}{\pricemdeltai{\delta}{j}}$.
Using this inclusion and the identity 
$
A\cap B
\subseteq
 (C\cap B) \cup (A\setminus C)
$, we find 
\[
\begin{array}{l}
\Intervali{j}{\profile} 
\refereq{\eqref{intersections:a}}{=} 
\fullIntervali{j}{\profile} \cap
\sublevelseti{j}{\pricei{j}} 
\subseteq
\fullIntervali{j}{\profile} \cap
\sublevelseti{j}{\pricemdeltai{\delta}{j}} 
\qquad\\\qquad\qquad\hfill
\subseteq
\big(\fullIntervali{j}{\profilemdelta{\delta}}\cap
\sublevelseti{j}{\pricemdeltai{\delta}{j}}\big)\cup\big(\fullIntervali{j}{\profile}\setminus\fullIntervali{j}{\profilemdelta{\delta}}\big)
=
\Intervali{j}{\profilemdelta{\delta}}\cup\big(\fullIntervali{j}{\profile}\setminus\fullIntervali{j}{\profilemdelta{\delta}}\big)
\, .
\end{array}
\]
It follows 
from~$
\xMeasure{\fullIntervali{j}{\profile}\setminus \fullIntervali{j}{\profilemdelta{\delta}}} \leq \xmeasuremax\lipschitz \delta
$
that
$
\bigxmeasure{\Intervali{j}{\profile}}
\leq 
\bigxmeasure{
\Intervali{j}{\profilemdelta{\delta}}}
+
\bigxmeasure{\fullIntervali{j}{\profile}\setminus\fullIntervali{j}{\profilemdelta{\delta}}}
\leq
\bigxmeasure{
\Intervali{j}{\profilemdelta{\delta}}}
+ \xmeasuremax\lipschitz \delta
$, which is the first inequality.  
Similarly, using the same identity twice, we find
\[
\begin{array}{l}
\Intervali{j}{\profilemdelta{\delta}} 
\refereq{\eqref{intersections:a}}{=}
\fullIntervali{j}{\profilemdelta{\delta}} \cap
\sublevelseti{j}{\pricemdeltai{\delta}{j}} 
\subseteq
\big(\fullIntervali{j}{\profile}\cap\sublevelseti{j}{\pricemdeltai{\delta}{j}}
\big)\cup\big(\fullIntervali{j}{\profilemdelta{\delta}}\setminus\fullIntervali{j}{\profile}\big)
\\\qquad\qquad\quad
=
\big(\sublevelseti{j}{\pricemdeltai{\delta}{j}}
\cap\fullIntervali{j}{\profile}\big)\cup\big(\fullIntervali{j}{\profilemdelta{\delta}}\setminus\fullIntervali{j}{\profile}\big)
\\\qquad\qquad\qquad\quad
\subseteq
\big(\sublevelseti{j}{\pricei{j}}\cap\fullIntervali{j}{\profile}
\big)
\cup
\big(\sublevelseti{j}{\pricemdeltai{\delta}{j}}\setminus\sublevelseti{j}{\pricei{j}}\big)
\cup
\big(\fullIntervali{j}{\profilemdelta{\delta}}\setminus\fullIntervali{j}{\profile}\big)
\\\qquad\qquad\qquad\qquad\qquad\qquad
=
\Intervali{j}{\profile}
\cup
\big(\sublevelseti{j}{\pricemdeltai{\delta}{j}}\setminus\sublevelseti{j}{\pricei{j}}\big)
\cup
\big(\fullIntervali{j}{\profilemdelta{\delta}}\setminus\fullIntervali{j}{\profile}\big)
\, .
\end{array}
\]
It follows 
from $
\xMeasure{\fullIntervali{j}{\profilemdelta{\delta}}\setminus\fullIntervali{j}{\profile}}
\leq
\xmeasuremax\lipschitz \delta
$
and \cref{lemma:sublevelsets}
that
$
\bigxmeasure{\Intervali{j}{\profilemdelta{\delta}}}
\leq
\bigxmeasure{\Intervali{j}{\profile}}
+
\bigxmeasure{\fullIntervali{j}{\profilemdelta{\delta}}\setminus\fullIntervali{j}{\profile}}
+
\bigxmeasure{\sublevelseti{j}{\pricemdeltai{\delta}{j}}\setminus\sublevelseti{j}{\pricei{j}}}
\leq
\bigxmeasure{\Intervali{j}{\profile}}
+
\xmeasuremax\lipschitz \delta
+
2\xmeasuremax \sqrt{ 2\delta/\stronglyconvex }  
$, which is the second inequality. 
\qed\end{proof}

\noindent
The~$\xmeasuremax\lipschitz \delta$ terms 
in \cref{lemma:doubleinequalities-part-two:new}
are due to deviations of the preference intervals in the modified problem~\eqref{upper-level-alter-delta} after discretization of the prices. An additional term~$2\xmeasuremax\sqrt{2 \delta/\stronglyconvex}$ appears in 
the second inequality
because rounded down prices in~\eqref{upper-level-alter-delta} may attract users that refuse service in~\eqref{upper-level}.
Note that in the particular case~$\profile\in(\delta\Integer
)^n$, we have $\Intervali{j}{\profilemdelta{\delta}}
=
\Intervali{j}{\profile} 
$ and thus $\xMeasure{\Intervali{j}{\profilemdelta{\delta}}}
=
\xMeasure{\Intervali{j}{\profile}} 
$.

Our proof for \cref{prop:bounded-delta} follows from \cref{lemma:doubleinequalities-part-two:new}.

\begin{proof}[
\cref{prop:bounded-delta}]
Following the lines of the proof of \cref{prop:bounded}, we can show that~\eqref{upper-level-alter-delta} admits an optimal solution~$\profile$ in $[\pricemindelta{\delta}-\delta,\pricemax+\delta]^n$.
Since by construction
$\altprofile=(\pricemdeltai{\delta}{1},\dots,\pricemdeltai{\delta}{n})$
satisfies $\Intervali{j}{\profilemdelta{\delta}} = \Intervali{j}{\altprofilemdelta{\delta}}$ for~$j\in[n]$, we find that
$\altprofile\in\big([\pricemindelta{\delta}-\delta,\pricemax+\delta]\cap\delta\Integer\big)^n$ is an optimal solution of~\eqref{upper-level-alter-delta} too. 
Hence $\OPTdelta{\delta}{\Real}=\OPTdelta{\delta}{\delta\Integer}$.

It remains to show the inequality $\OPTdelta{\delta}{\delta\Integer} 
+n\xmeasuremax\lipschitz \delta\pricemax
\geq \OPT{\Real}$.
Let $\profile\in[\pricemin,\pricemax]^n$ be an optimal solution of~\eqref{upper-level}, which exists according to \cref{prop:bounded}, and consider its discrete approximation $\profilemdelta{\delta}$. 
Using the second inequality in \cref{lemma:doubleinequalities-part-two:new}, we find $\xMeasure{\Intervali{j}{\profilemdelta{\delta}}}
\leq
\xMeasure{\Intervali{j}{\profile}} + \xmeasuremax\lipschitz \delta 
+ 2\xmeasuremax\sqrt{2 \delta /\stronglyconvex} 
\leq 
\ymeasurei{j} + \xmeasuremax\lipschitz \delta 
+ 2\xmeasuremax\sqrt{2 \delta /\stronglyconvex} 
$ for all $j\in[n]$, which implies that~$\profile$ is a feasible solution of~\eqref{upper-level-alter-delta}. 
Using $\pricemdeltai{\delta}{j} + \delta \geq \pricei{j} $ and the first inequality in \cref{lemma:doubleinequalities-part-two:new}, we find 
$
\OPTdelta{\delta}{\delta\Integer} 
\geq
\sum_{j=1}^n (\pricemdeltai{\delta}{j} + \delta )
\xmeasure{\Intervali{j}{\profilemdelta{\delta}}}
\geq 
\sum_{j=1}^n \pricei{j}
\big(\xMeasure{\Intervali{j}{\profile}} 
- \xmeasuremax\lipschitz \delta\big)
\geq 
\OPT{\Real} - n\xmeasuremax\lipschitz \delta\pricemax$.
\qed\end{proof}

\subsubsection{Proof of \cref{prop:upper_bound}} 
\label{section:monitoring}

\noindent
The proof of \cref{prop:upper_bound} requires a technical result.

\renewcommand{\slot}{j}%
\renewcommand{\altslot}{i}%
\renewcommand{\altaltslot}{k}%
\begin{lemma}\label{lemma:reduction}
\renewcommand{\altprofile}{\altpricefunction}%
\renewcommand{\altpricefunction}{p'}%
There exists $\cstgain>0$ such that, for all~$\profile\in
[\pricemindelta{\delta},+\infty)
^n$, $\slot\in[n]$, and $\Delta\geq 0$, the following holds: 
\begin{easylist}
\ListProperties(Numbers=r,Mark=.,FinalMark={)},FinalMark3={)})
\item $\bullet$
$[\xvar,\xvar+ \cstgain\Delta]\subseteq\Intervali{\slot}{\profile}$ for all $\xvar\in\bigIntervali{\slot}{(\pricei{1}{\,+\,}\Delta,\dots,\pricei{\slot-1}{\,+\,}\Delta,\pricei{\slot}{\,+\,}\Delta,\pricei{\slot+1},\dots,\pricei{n})}$. 
\item $\bullet$
$[\xvar- \cstgain\Delta,\xvar ]\subseteq\Intervali{\slot}{\profile}$ for all $\xvar\in\bigIntervali{\slot}{(\pricei{1},\dots,\pricei{\slot-1},\pricei{\slot}{\,+\,}\Delta,\pricei{\slot+1}{\,+\,}\Delta,\dots,\pricei{n}{\,+\,}\Delta)}$.
\end{easylist}
\end{lemma}
\begin{proof}
\renewcommand{\Idist}{S}%
\providecommand{\Idistij}[2]{\Idist_{#1,#2}}%
\renewcommand{\altprofile}{\altpricefunction}%
\renewcommand{\altpricefunction}{p'}%
Let~$\profile\in\Real^n$, $\slot\in[n]$, and $\Delta\geq 0$. We only derive the result for~$\Intervali{\slot}{\altprofile}$ with $\altprofile = (\pricei{1}+\Delta,\dots,\pricei{\slot-1}+\Delta,\pricei{\slot}+\Delta,\pricei{\slot+1},\dots,\pricei{n})$. 
A proof for~$\Intervali{\slot}{\altprofile}$ with $\altprofile = (\pricei{1},\dots,\pricei{\slot-1},\pricei{\slot}+\Delta,\pricei{\slot+1}+\Delta,\dots,\pricei{n}+\Delta)$ is obtained similarly.

Consider $\Idist =  \{ \xvar \in\Real \colon \dist{\xvar} + \pricemin \leq 0 \}$.  
If $\pricemin=\pricemax$, then $\Idist=\{0\}$ and we have either~$\Delta=0$, or $\altpricei{\slot}>\pricemax$ and $\Intervali{\slot}{\altprofile}=\varnothing$.
Since the result is immediate in both cases, we can assume that~$\Idist$ is a nondegenerate interval.
Let~$\IdistLipschitz$ be the Lipschitz constant of~$\distfunction$ over~$\Idist$, 
and let~$\Smooth$ be the maximum Lipschitz constant among the functions~$\xvar\mapsto\dist{\xvar-\yvari{i}}-\dist{\xvar-\yvari{j}}$ over the intervals~$
 \Idistij{i}{j}=\{ \xvar+\yvar \colon \xvar\in\Idist,\, \yvar\in [\yvari{i},\yvari{j}] \}
$ for $i,j\in[n]$ with~$i<j$. 
We show the result for $\cstgain=
1/\max(\IdistLipschitz,\Smooth)$. 

Let~$\xvar\in\Intervali{\slot}{\altprofile}$. It follows from~\eqref{intersections:a} that 
$\xvar
\in\fullIntervali{\slot}{\altprofile}$ and $\xvar\in\sublevelseti{\slot}{\altpricei{\slot}}$.
Since $\xvar\in\sublevelseti{\slot}{\altpricei{\slot}}$, we have 
$
\dist{\xvar-\yvari{\slot}} +\pricemin
\leq
\dist{\xvar-\yvari{\slot}} +\pricei{\slot}
=
\dist{\xvar-\yvari{\slot}} +\altpricei{\slot}-\Delta
\leq
-\Delta
$, and~$\Delta\geq 0$ yields $\xvar-\yvari{\slot}\in\Idist$. 
Since the result is immediate when $\Intervali{\slot}{\altprofile}=\varnothing$, we assume both $\fullIntervali{\slot}{\altprofile}\neq\varnothing$ and $\sublevelseti{\slot}{\altpricei{\slot}}\neq\varnothing$, and we show 
that, for all $\xvar\in \Intervali{\slot}{\altprofile}$ and $\altxvar\in  [\xvar,\xvar+ \cstgain\Delta]$, we have
$\altxvar\in\Intervali{\slot}{\profile}$.

We first establish the preliminary property that every~$\altxvar\in[\xvar - \Delta/\IdistLipschitz,\xvar + \Delta/\IdistLipschitz]$ satisfies 
$\altxvar-\yvari{\slot}\in\Idist$ and $\altxvar\in\Idistij{\slot}{\altaltslot}$ for all~$k\in[n]$ with $\slot<\altaltslot$.
%
Consider any~$\altxvar\in[\xvar - \Delta/\IdistLipschitz,\xvar + \Delta/\IdistLipschitz]$ and suppose $
\dist{\altxvar-\yvari{\slot}} +\pricemin
> 0$. By continuity of~$\distfunction$, we can find~$\altaltxvar\in(\xvar - \Delta/\IdistLipschitz,\xvar + \Delta/\IdistLipschitz)$ such that $
\dist{\altaltxvar-\yvari{\slot}} +\pricemin
= 0$, which implies $\altaltxvar-\yvari{\slot}\in\Idist$.
It follows from $\xvar-\yvari{\slot}\in\Idist$, the convexity of~$\Idist$, and the definition of~$\IdistLipschitz$ that $
\dist{\altaltxvar-\yvari{\slot}} +\pricemin
\leq 
\dist{\xvar-\yvari{\slot}} +\pricemin
+ \IdistLipschitz |\altaltxvar-\xvar|
\leq 
- \Delta + \IdistLipschitz |\altaltxvar-\xvar|
<0
$, which is a contradiction. 
Hence, 
$
\dist{\altxvar-\yvari{\slot}} +\pricemin
\leq 0
$, which
implies the desired property.

For $\altxvar\in[\xvar,\xvar + \cstgain\Delta]$, we have~$\altxvar\in[\xvar - \Delta/\IdistLipschitz,\xvar + \Delta/\IdistLipschitz]$, and the preliminary property yields
$\altxvar-\yvari{\slot}\in\Idist$.
It follows from $\xvar-\yvari{\slot}\in\Idist$, the convexity of~$\Idist$, and the definition of~$\IdistLipschitz$ that $
\dist{\altxvar-\yvari{\slot}} + \pricei{\slot}
\leq 
\dist{\xvar-\yvari{\slot}} + \IdistLipschitz |\altxvar-\xvar| + \pricei{\slot}  
\leq 
\dist{\xvar-\yvari{\slot}} + \IdistLipschitz\cstgain\Delta + \pricei{\slot}  
\leq 
\dist{\xvar-\yvari{\slot}} + \Delta + \pricei{\slot}  
= 
\dist{\xvar-\yvari{\slot}} + \altpricei{\slot} 
\leq 
0$. Thus, for $\altxvar\in[\xvar,\xvar + \cstgain\Delta]$, we have $\altxvar\in\sublevelseti{\slot}{\pricei{\slot}}$.

Now, consider $\altaltslot\in[n]$. First suppose $\altaltslot \leq \slot$.
For $\altxvar\geq\xvar$,  \cref{fact:increasing} 
yields $  \dist{\altxvar-\yvari{\altaltslot}} - \dist{\altxvar-\yvari{\slot}} \geq \dist{\xvar-\yvari{\altaltslot}} - \dist{\xvar-\yvari{\slot}}$. It follows that
\[
\dist{\altxvar-\yvari{\slot}} + \pricei{\slot} 
\leq 
\dist{\altxvar-\yvari{\altaltslot}} + \pricei{\altaltslot} + [\dist{\xvar-\yvari{\slot}} + \altpricei{\slot} ] - [ \dist{\xvar-\yvari{\altaltslot}} +\altpricei{\altaltslot} ] 
\leq 
\dist{\altxvar-\yvari{\altaltslot}} + \pricei{\altaltslot}
\, ,
\]
where we use $
\pricei{\slot}
=
\pricei{\altaltslot}+\altpricei{\slot}-
\altpricei{\altaltslot}
$ and $\xvar\in\fullIntervali{\slot}{\altprofile}$.
Suppose now $\altaltslot > \slot$.
For $\altxvar\in[\xvar,\xvar + \cstgain\Delta]$, we have~$\altxvar\in[\xvar - \Delta/\IdistLipschitz,\xvar + \Delta/\IdistLipschitz]$, and 
the preliminary property yields
$\altxvar\in\Idistij{\slot}{\altaltslot}$.
It follows from $\xvar\in\Idistij{\slot}{\altaltslot}$, the convexity of~$\Idistij{\slot}{\altaltslot}$, and the definition of~$\Smooth$ that $  \dist{\altxvar-\yvari{\slot}} - \dist{\altxvar-\yvari{\altaltslot}} 
\leq 
\dist{\xvar-\yvari{\slot}} - \dist{\xvar-\yvari{\altaltslot}} + \Smooth (\altxvar-\xvar) 
\leq 
\dist{\xvar-\yvari{\slot}} - \dist{\xvar-\yvari{\altaltslot}} + \Smooth\cstgain \Delta
\leq 
\dist{\xvar-\yvari{\slot}} - \dist{\xvar-\yvari{\altaltslot}} + \Delta$. Consequently,
\[
\dist{\altxvar-\yvari{\slot}} + \pricei{\slot} 
\leq 
\dist{\altxvar-\yvari{\altaltslot}} + \pricei{\altaltslot} + [ \dist{\xvar-\yvari{\slot}}+ \altpricei{\slot} ] - [\dist{\xvar-\yvari{\altaltslot}} + \altpricei{\altaltslot}]
\leq
\dist{\altxvar-\yvari{\altaltslot}} + \pricei{\altaltslot}
\, ,
\]
where we use $
\pricei{\slot}=\pricei{\altaltslot} + \altpricei{\slot}-\altpricei{\altaltslot}-\Delta
$ and $\xvar\in\fullIntervali{\slot}{\altprofile}$.
Thus, for $\altxvar\in[\xvar,\xvar + \Delta/\Smooth]$ we have $
\dist{\altxvar-\yvari{\slot}} + \pricei{\slot} 
\leq
\dist{\altxvar-\yvari{\altaltslot}} + \pricei{\altaltslot}
$ for all $\altaltslot\in[n]$, and $\altxvar\in\fullIntervali{\slot}{\profile}$.

Finally, combining our last two conclusions and using~\eqref{intersections:a}, we find that $\altxvar
\in\fullIntervali{\slot}{\profile}\cap\sublevelseti{\slot}{\pricei{\slot}}=\Intervali{\slot}{\profile}$ holds
for all $\xvar\in
\Intervali{\slot}{\altprofile}$ and all $\altxvar\in
[\xvar,\xvar+ \cstgain\Delta]$.
 \qed\end{proof}

In order to show \cref{prop:upper_bound}, we proceed as follows. We consider a discrete optimal solution~$\profile$ of~\eqref{upper-level-alter-delta} in accordance with~\cref{lem:upper}, and then we alter the individual prices one time slot at a time until a new discrete price profile is obtained that satisfies the constraints of~\eqref{upper-level}, thus providing us with a lower bound to~$\OPT{\delta\Integer}$. To do so, we use two main tools: \cref{lemma:sponge-slot} identifies for~$\profile$ a particular time slot~$\spongeslot$ where a substantial decrease in service load~$\Intervalifunction{\spongeslot}$ can be obtained through simultaneous increase of all prices, and \cref{lemma:tool} provides bounds 
for the variations in service load at the time slots resulting from the successive price alterations.
Starting from~$\profile$, our strategy is then to progressively reduce service load though successive applications of \cref{lemma:tool} from the extreme time slots towards the specified slot~$\spongeslot$, where initially the service load had been sufficiently reduced so as to compensate for the cumulated side effects of the future price jumps at the other time slots.

We set $
\deltamax = 
8/(\lipschitz^2\stronglyconvex)
$, as this constant plays a role in the upcoming results.

\newcommand{\deltaboundfunction}{\bar{\delta}}
\newcommand{\deltabound}[1]{\deltaboundfunction(#1)}%
\newcommand{\IdistdeltaLipschitz}{\IdistLipschitz}%
\newcommand{\Idistdelta}{\Idist}%
\begin{lemma}\label{lemma:sponge-slot}
 Let $\csth >0$. There exist 
 $
\cstDelta{\csth}>0
$
and
 $
0<\deltabound{\csth}\leq \deltamax
$
such that, for every $\delta\in(0,\deltabound{\csth}
 ]$ and  every~$\profile\in\big(
[\pricemindelta{\delta}-\delta,\pricemax+\delta]\cap\delta\Integer\big)^n$ that is an optimal solution of~\eqref{upper-level-alter-delta}, setting $\MCspongefunction= \big\lceil 
\cstDelta{\csth} \delta^{-{3}/{4}}\big\rceil$ yields one of the following: 
$
\sum_{\slot=1}^{n}
\xmeasure{\Intervali{\slot}{\profile+\MCspongefunction\delta}}
=
0
$, 
or one can find a time slot~$\spongeslot\in[n]$ such that
$
\xmeasure{\Intervali{\spongeslot}{\profile+\MCspongefunction\delta}}
\leq
\xMeasure{\Intervali{\spongeslot}{\profile}}  
-  \csth \sqrt{\delta}
$.
\end{lemma}

In~\cref{lemma:sponge-slot} and its proof, we use the notation
$\profile+\MCspongefunction\delta=(\pricei{1}+\MCspongefunction\delta,\dots,\pricei{n}+\MCspongefunction\delta)$.
We note that the existence of an optimal solution~$\profile\in\big(
[\pricemindelta{\delta}-\delta,\pricemax+\delta]
\cap\delta\Integer\big)^n$  of the modified problem~\eqref{upper-level-alter-delta} follows from \cref{prop:bounded-delta}.
The message of~\cref{lemma:sponge-slot} is that the solution~$\profile$ of~\eqref{upper-level-alter-delta} is either~$O(\delta^{1/4})
$ close to a feasible profile~$\altprofile\in(\delta\Integer)^n$ of~\eqref{upper-level} under which no user decides to be served, or that there is one particular time slot~$\spongeslot$ where any overhead~$\csth \sqrt{\delta}$ in service load can be cut by simultaneously increasing all prices by a quantity~$\MCspongefunction\delta
=O(\delta^{1/4})
$.

\begin{proof}[
\cref{lemma:sponge-slot}]
\label{proof:lemma:sponge-slot}%
\renewcommand{\Idist}{S'}
\renewcommand{\IdistdeltaLipschitz}{\IdistLipschitz'}%
\newcommand{\positivebound}{\pi}%
\newcommand{\Diff}{D}%
\newcommand{\Diffi}[1]{\Diff_{#1}}%
\newcommand{\altpricepdeltaimax}{\pi}%
We show the lemma for the maps
$
\cstDelta{\csth}
= 
[ 
n\positivebound(\IdistdeltaLipschitz/\xmeasuremin)\csth  
]^{ 1 / 2}
$
and
$
\deltabound{\csth} 
= 
\min \big(\deltamax, [\IdistdeltaLipschitz\xmeasure{\Real}/(\xmeasuremin\cstDelta{\csth})]^{4} \big)
$,
where 
$ 
\positivebound=  \max(0,\pricemax+ 2\deltamax)+\pricemax-\pricemindelta{\deltamax} + 
\deltamax>0
$ 
and~$\IdistdeltaLipschitz$ denotes the Lipschitz constant of~$\distfunction$ over the interval~$\Idistdelta=\{ \xvar \in\Real \colon \dist{\xvar} + 
\pricemindelta{\deltamax} - \deltamax
\leq 0 \}
$.
Consider~$\delta$, $\profile$, and~$\MCspongefunction$ as introduced in the statement of the lemma, and set
$
 \loadmax 
 = 
(\xmeasuremin/\IdistdeltaLipschitz)\cstDelta{\csth}\sqrt[4]{\delta}
$.

First consider the case 
$ \MCspongefunction\delta> \pricemax-\pricemindelta{\deltamax} + 
\delta$. 
By definition of~$\profile$  and~$\delta$, we have $\pricei{\slot}+\MCspongefunction\delta> \pricemax
$ for all~$\slot\in[n]$, and therefore no user is served in~\eqref{upper-level} under~$\profile+\MCspongefunction\delta$. Hence,
$
\sum_{\slot=1}^{n}
\xmeasure{\Intervali{\slot}{\profile+\MCspongefunction\delta}}
=
0
$, which is the first alternative.

Now, consider the case $
\sum_{\slot=1}^{n}
\xmeasure{\Intervali{j}{\profile}}
< 
\loadmax 
$. 
We prove that the first alternative $
\sum_{\slot=1}^{n}\xmeasure{\Intervali{\slot}{\profile+\MCspongefunction\delta}}
=
0
$ holds, once again by showing that no user is served under~$\profile+\MCspongefunction\delta$.
To do this, we take $\xvar\in\supp(\xmeasurefunction)$ and ${\altaltslot\in[n]}$, for which we show that $\dist{\xvar-\yvari{\altaltslot}}+\pricei{\altaltslot} +\MCspongefunction\delta  > 0$. %
Since $\dist{\xvar-\yvari{\altaltslot}}+\pricei{\altaltslot} +\MCspongefunction\delta  > 0$ is immediate when $\dist{\xvar-\yvari{\altaltslot}}+\pricei{\altaltslot} > 0$, we can suppose that 
$\dist{\xvar-\yvari{\altaltslot}}+\pricei{\altaltslot} \leq 0$, which implies $\xvar-\yvari{\altaltslot}\in\Idistdelta$. 
Since $
\delta
\leq
\deltabound{\csth} 
\leq 
[\IdistdeltaLipschitz\xmeasure{\Real}/(\xmeasuremin\cstDelta{\csth})]^{4} 
$, we have $ 
\loadmax 
\leq
\xmeasure{\Real}
$,
and 
therefore $
\sum_{\slot=1}^{n}\xmeasure{\Intervali{\slot}{\profile}} 
< \xmeasure{\Real}
$%
. 
Hence, one can find~$\altxvar\in\supp(\xmeasurefunction)$ such that $\dist{\altxvar-\yvari{\slot}}+\pricei{\slot}>0$ for all~$\slot\in[n]$.
In particular, $ \dist{\altxvar-\yvari{\altaltslot}}+\pricei{\altaltslot}  
>0$, and it follows from the continuity of~$\distfunction$ and the convexity of~$\supp(\xmeasurefunction)$ that there exists~$\altaltxvar\in\supp(\xmeasurefunction)$ such that $\dist{\altaltxvar-\yvari{\altaltslot}}+\pricei{\altaltslot}=0$, $\altaltxvar-\yvari{\altaltslot}\in\Idistdelta$, and $\dist{\hat\xvar-\yvari{\altaltslot}}+\pricei{\altaltslot} \leq 0$
for all~$\hat\xvar$ located between~$\xvar$ and~$\altaltxvar$. This last property yields  $\hat\xvar\in\bigcup_{\slot=1}^n \Intervali{\slot}{\profile}$ for all such~$\hat\xvar$, and consequently $|\xvar-\altaltxvar| \leq   \sum_{\slot=1}^n |\Intervali{\slot}{\profile}|
$.  
Since $
\MCspongefunction\delta 
\geq 
{\IdistdeltaLipschitz \loadmax}/{\xmeasuremin} 
$, 
we have
$
\MCspongefunction\delta >
(\IdistdeltaLipschitz/{\xmeasuremin})
\sum_{\slot=1}^{n}\xmeasure{\Intervali{\slot}{\profile}} 
\geq
\IdistdeltaLipschitz \sum_{\slot=1}^n |\Intervali{\slot}{\profile}|
\geq
\IdistdeltaLipschitz |\xvar-\altaltxvar|
$%
. 
Consequently,  $\dist{\xvar-\yvari{\altaltslot}}+\pricei{\altaltslot}+\MCspongefunction\delta 
\geq 
\dist{\altaltxvar-\yvari{\altaltslot}}+\pricei{\altaltslot} - \IdistdeltaLipschitz |\xvar-\altaltxvar| +\MCspongefunction\delta
>
\dist{\altaltxvar-\yvari{\altaltslot}}+\pricei{\altaltslot}
=
0
$, where the first inequality results from the definition of~$\IdistdeltaLipschitz$.

Lastly, consider the case 
$ \MCspongefunction\delta \leq \pricemax-\pricemindelta{\deltamax} +
\delta$ and  
$%
\sum_{\slot=1}^{n}
\xmeasure{\Intervali{j}{\profile}}
\geq \loadmax
$. 
Let the price profile $\altprofile$ be defined for $\slot\in[n]$ by
$
\altpricei{\slot} = \pricei{\slot}+\MCspongefunction\delta
$, and 
let $
\Diffi{\slot}
=
\xmeasure{\Intervali{\slot}{\profile}} - 
\xmeasure{\Intervali{\slot}{\altprofile}}
$. 
By definition of~$\profile$ and~$\altprofile$, we have $
\fullIntervali{\slot}{\profile}
=
\fullIntervali{\slot}{\altprofile}
$
and 
$
\sublevelseti{\slot}{\pricei{\slot}}
\supseteq
\sublevelseti{\slot}{\altpricei{\slot}}
$.
It follows from~\eqref{intersections:a} 
that 
$
\Intervali{\slot}{\profile}
\supseteq
\Intervali{\slot}{\altprofile}
$, and therefore $\Diffi{\slot}\geq 0$. 
Optimality of~$\profile$ in~\eqref{upper-level-alter-delta} and~$\profile,\altprofile\in(\delta\Integer)^n$ then imply
\[
 0
 \geq
 \sum_{\slot=1}^n 
 (
 \altpricei{\slot} 
 + \delta )
 \xmeasure{\Intervali{\slot}{\altprofile}} 
 -
 \sum_{\slot=1}^n 
 (
 \pricei{\slot}
 +\delta)
\xmeasure{\Intervali{\slot}{\profile}}
 =
\MCspongefunction\delta
 \sum_{\slot=1}^n 
\xmeasure{\Intervali{\slot}{\profile}}
 -
 \sum_{\slot=1}^n
 (
 \altpricei{\slot} 
 +\delta)
 \Diffi{\slot}
\, .
\]
Using $\delta\leq\deltamax$, the definition of~$\profile$ and~$\altprofile$, and the assumption on~$\MCspongefunction\delta$, we find
$
\altpricei{\slot}
+\delta
= 
\pricei{\slot}+\delta+\MCspongefunction\delta
\leq
(\pricemax+\delta)+\delta+ (\pricemax-\pricemindelta{\deltamax} +
\delta)
\leq
\positivebound$ for all~$\slot\in[n]$.
It follows under our assumption that
\[
 \loadmax
 \leq
 \sum_{\slot=1}^n 
\xmeasure{\Intervali{\slot}{\profile}}
 \leq
 \frac{1}{\MCspongefunction\delta} \sum_{\slot=1}^n 
 (
 \altpricei{\slot}
 +\delta)
 \Diffi{\slot}
 \leq
 \frac{
 \positivebound
 }{\MCspongefunction\delta}  \sum_{\slot=1}^n  \Diffi{\slot}
 \leq
 \frac{n   \positivebound
}{\MCspongefunction\delta}  \max_{\slot\in[n]}  \Diffi{\slot}
\, .
\]
Consequently, we can find~$\spongeslot\in[n]$ such that
$
 \Diffi{\spongeslot}
 \geq
 \loadmax
 \MCspongefunction\delta/(n \positivebound
)  
$.
It follows from $\profile,\altprofile\in(\delta\Integer)^n$
that
$
\xmeasure{\Intervali{\spongeslot}{\altprofile}}
=
\xmeasure{\Intervali{\spongeslot}{\profile}} 
-
\Diffi{\spongeslot}
\leq
\xmeasure{\Intervali{\spongeslot}{\profile}} 
-
\loadmax
\MCspongefunction\delta/(n \positivebound)
$.
Introducing the expressions for~$ \loadmax$ and~$\cstDelta{\csth}$ into the last inequality and using 
$\MCspongefunction \geq
\cstDelta{\csth} \delta^{-{3}/{4}}$,
we find
$
\xmeasure{\Intervali{\spongeslot}{\altprofile}}
\leq
\xMeasure{\Intervali{\spongeslot}{\profile}}  -  \csth \sqrt{\delta}
$,
and we get the second alternative.
\qed\end{proof}

Given a price profile, suppose we simultaneously raise the prices at the first~$j$ (or last~$j$) time slots 
in order to reduce attendance at those time slots.
\cref{lemma:tool} provides bounds for the decrease in their individual service loads, and bounds for the variations in service load at the remaining time slots.

\begin{lemma}\label{lemma:tool}
\renewcommand{\altprofile}{\altpricefunction}%
\renewcommand{\altpricefunction}{p'}%
There exists $\MCtool>0$
such that for every $\delta\in(0,\deltamax]$, $\profile\in
\big([\pricemindelta{\delta},+\infty)
\cap\delta\Integer
\big)^n$,
and $\spongeslot\in[n]$ with~$\xmeasure{\fullIntervali{\spongeslot}{\profile}}>0
$, the following holds:

\begin{easylist}
\ListProperties(Numbers=r,Mark=.,FinalMark={)},FinalMark3={)})
\begin{enumerate}[label=(\roman*)\!]
 \item 
\label{lemma:tool:left}
 For $\slot{\,<\,}\spongeslot$, $\Delta{\,>\,}0$, and $\altprofile {\,=\,} (\pricei{1}+\MCtool\Delta,\dots,\pricei{\slot-1}+\MCtool\Delta,\pricei{\slot}+\MCtool\Delta,\pricei{\slot+1},\dots,\pricei{n})$, we have $
 \xmeasure{\fullIntervali{\spongeslot}{\altprofile}}>0
$, and
 \[
 \nocolsep
  \begin{array}{rcll}
   \xmeasure{\Intervali{\altaltslot}{\altprofile}} 
   &\leq&
   \xmeasure{\Intervali{\altaltslot}{\profile}}
   &
   \text{ if }
   1 \leq \altaltslot < \slot \, ,
   \\
   \xmeasure{\Intervali{\slot}{\altprofile}} 
   &\leq&
   \max\big(0,\xmeasure{\Intervali{\slot}{\profile}} - \Delta\big)
    \, , &
   \\
    \xmeasure{\Intervali{\altaltslot}{\altprofile}} 
   &\leq&
    \xmeasure{\Intervali{\altaltslot}{\profile}} +
    \xmeasuremax\MCtool\lipschitz
    \Delta
   &    \text{ if }
    \slot< \altaltslot \leq \spongeslot
    \, ,
   \\
   \xmeasure{\Intervali{\altaltslot}{\altprofile}} 
   &=&
   \xmeasure{\Intervali{\altaltslot}{\profile}}
   &
   \text{ if }
   \spongeslot < \altaltslot \leq n \, .
  \end{array}
 \]

 \item 
\label{lemma:tool:right}
 For $\slot{\,>\,}\spongeslot$, $\Delta{\,>\,}0$, and $\altprofile {\,=\,} (\pricei{1},\dots,\pricei{\slot-1},\pricei{\slot}+\MCtool\Delta,\pricei{\slot+1}+\MCtool\Delta,\dots,\pricei{n}+\MCtool\Delta)$, we have $
 \xmeasure{\fullIntervali{\spongeslot}{\altprofile}}>0
$, and
 \[
 \nocolsep
  \begin{array}{rcll}
   \xmeasure{\Intervali{\altaltslot}{\altprofile}} 
   &=&
   \xmeasure{\Intervali{\altaltslot}{\profile}}
   &
   \text{ if }
   1 \leq \altaltslot < \spongeslot \, ,
   \\
    \xmeasure{\Intervali{\altaltslot}{\altprofile}}
   &\leq&
   \xmeasure{\Intervali{\altaltslot}{\profile}} + 
   \xmeasuremax\MCtool\lipschitz
   \Delta
   & 
   \text{ if }
   \spongeslot\leq \altaltslot <\slot
    \, ,\\
   \xmeasure{\Intervali{\slot}{\altprofile}}
   &\leq& 
   \max\big(0,\xmeasure{\Intervali{\slot}{\profile}} - \Delta\big)
    \, , &
   \\
   \xmeasure{\Intervali{\altaltslot}{\altprofile}}
   &\leq&
   \xmeasure{\Intervali{\altaltslot}{\profile}}
   &
   \text{ if }
   \slot < \altaltslot \leq n \, .
  \end{array}
 \]
\end{enumerate}
\end{easylist}
\end{lemma}
\begin{proof}
\label{proof:lemma:tool}
\renewcommand{\altprofile}{\altpricefunction}%
\renewcommand{\altpricefunction}{p'}%
We show~\ref{lemma:tool:left} for the constants
$\MCtool 
= 
\lceil 1/(\xmeasuremin\cstgain)\rceil
$, where~$\cstgain$ is specified by \cref{lemma:reduction}. 
The proof of~\ref{lemma:tool:right} is omitted as it can be done similarly and with the same constants.
Let~$\slot<\spongeslot$, $\Delta>0$, and $\altprofile = (\pricei{1}+\MCtool\Delta,\dots,\pricei{\slot-1}+\MCtool\Delta,\pricei{\slot}+\MCtool\Delta,\pricei{\slot+1},\dots,\pricei{n})$. 
First consider~$\xvar\in\fullIntervali{\spongeslot}{\profile}$.
For $\ell\in[n]$, we have
$\pricei{\ell} 
\leq\altpricei{\ell}$ and $\pricei{\spongeslot}=\altpricei{\spongeslot}$, which yield 
$
\dist{\xvar-\yvari{\spongeslot}} + \altpricei{\spongeslot} 
=
\dist{\xvar-\yvari{\spongeslot}} + \pricei{\spongeslot}  
\leq
\dist{\xvar-\yvari{\ell}} + \pricei{\ell} 
\leq
\dist{\xvar-\yvari{\ell}} + \altpricei{\ell} 
$, and therefore $\xvar \in \fullIntervali{\spongeslot}{\altprofile}$. Hence, $ \fullIntervali{\spongeslot}{\profile}\subseteq\fullIntervali{\spongeslot}{\altprofile}$. Since $
\xmeasure{\fullIntervali{\spongeslot}{\profile}}>0
$, we also have $
\xmeasure{\fullIntervali{\spongeslot}{\altprofile}}>0
$.
We now verify each remaining claim of~\ref{lemma:tool:left} individually.
Consider~$\altaltslot\in[n]$.

Suppose $1 \leq \altaltslot < \slot$. Then, for all $\xvar\in\Intervali{\altaltslot}{\altprofile}$ and $\ell\in [n]$, it holds that
$ 
\dist{\xvar-\yvari{\altaltslot}}+\pricei{\altaltslot} 
=
\dist{\xvar-\yvari{\altaltslot}}+\altpricei{\altaltslot} 
-
\MCtool\Delta
\leq
\min(0,\dist{\xvar-\yvari{\ell}}+\altpricei{\ell}) 
-
\MCtool\Delta
\leq
\min(0,\dist{\xvar-\yvari{\ell}}+\pricei{\ell}) 
$, so that 
$\xvar\in\Intervali{\altaltslot}{\profile}$. Hence, $\Intervali{\altaltslot}{\altprofile}\subseteq\Intervali{\altaltslot}{\profile}$, and $   \xmeasure{\Intervali{\altaltslot}{\altprofile}}\leq
   \xmeasure{\Intervali{\altaltslot}{\profile}} 
   $.

Suppose $\altaltslot=\slot$. %
In the case $\Intervali{\slot}{\altprofile}=\varnothing$, we have $\xmeasure{\Intervali{\slot}{\altprofile}}=0$.
Now, consider the case
$\Intervali{\slot}{\altprofile}\neq\varnothing$.
\cref{lemma:reduction} yields $[\altxvar,\altxvar+ \cstgain\MCtool\Delta]\subseteq\Intervali{\slot}{\profile}$ for all $\altxvar\in\Intervali{\slot}{\altprofile}$.
Since $\xmeasure{\fullIntervali{\spongeslot}{\profile}}>0$, $\slot<\spongeslot$, and~$\xmeasurefunction$ is interval supported, it follows from \cref{prop:intervals} that $  \xmeasure{\Intervali{\slot}{\profile}}\geq \xmeasure{\Intervali{\slot}{\altprofile}}+\xmeasuremin\cstgain\MCtool\Delta\geq \xmeasure{\Intervali{\slot}{\altprofile}}+\Delta$.
In any case, we have
$   \xmeasure{\Intervali{\slot}{\altprofile}} 
   \leq
    \max\big(0,\xmeasure{\Intervali{\slot}{\profile}} - \Delta\big)$.

Suppose $\slot<\altaltslot\leq\spongeslot$. 
Using \cref{lemma:twolemmasinone} with parameters~$\eta=\MCtool\Delta$ and~$\eta'=0$, we find
$
|\fullIntervali{\altaltslot}{\altprofile}\setminus\fullIntervali{\altaltslot}{\profile}|
\leq
\lipschitz\MCtool\Delta
$.
Since $\altpricei{\altaltslot}=\pricei{\altaltslot}$, we also have
$
\sublevelseti{\altaltslot}{\altpricei{\altaltslot}} 
=
\sublevelseti{\altaltslot}{\pricei{\altaltslot}} 
$.
It then follows from~\eqref{intersections:a} that
$
| \Intervali{\altaltslot}{\altprofile} \setminus \Intervali{\altaltslot}{\profile} | 
=
| (\fullIntervali{\altaltslot}{\altprofile} \setminus \fullIntervali{\altaltslot}{\profile}) \cap
\sublevelseti{\altaltslot}{\pricei{\altaltslot}} 
 | 
\leq
| \fullIntervali{\altaltslot}{\altprofile} \setminus\fullIntervali{\altaltslot}{\profile} | 
\leq 
\MCtool\lipschitz
$, 
and thus
$
\xmeasure{\Intervali{\altaltslot}{\altprofile}} 
   \leq
    \xmeasure{\Intervali{\altaltslot}{\profile}} + \xmeasure{ \Intervali{\altaltslot}{\altprofile} \setminus \Intervali{\altaltslot}{\profile}}
   \leq
    \xmeasure{\Intervali{\altaltslot}{\profile}} + \xmeasuremax| \Intervali{\altaltslot}{\altprofile} \setminus \Intervali{\altaltslot}{\profile}|
   \leq
    \xmeasure{\Intervali{\altaltslot}{\profile}} + 
    \xmeasuremax\MCtool\lipschitz
    \Delta 
    $.

Suppose $\spongeslot< \altaltslot \leq n$. We successively show $\Intervali{\altaltslot}{\profile}\subseteq\Intervali{\altaltslot}{\altprofile}$ and $\Intervali{\altaltslot}{\altprofile}\subseteq\Intervali{\altaltslot}{\profile}$, which imply $\Intervali{\altaltslot}{\profile}=\Intervali{\altaltslot}{\altprofile}$ and $\xmeasure{\Intervali{\altaltslot}{\profile}}=\xmeasure{\Intervali{\altaltslot}{\altprofile}}$. 
First, consider~$\xvar\in\Intervali{\altaltslot}{\profile}$.
For $\ell\in[n]$, we have 
$ 
\dist{\xvar-\yvari{\altaltslot}}+\altpricei{\altaltslot} 
=
\dist{\xvar-\yvari{\altaltslot}}+\pricei{\altaltslot} 
\leq
\min(0,\dist{\xvar-\yvari{\ell}}+\pricei{\ell})
\leq
\min(0,\dist{\xvar-\yvari{\ell}}+\altpricei{\ell}) 
$. Hence, 
$\xvar\in\Intervali{\altaltslot}{\altprofile}$, and therefore $\Intervali{\altaltslot}{\profile}\subseteq\Intervali{\altaltslot}{\altprofile}$.
Now, consider~$\xvar\in\Intervali{\altaltslot}{\altprofile}$, and pick~$\altxvar\in\fullIntervali{\spongeslot}{\profile}$.
    For $\ell\in[n]$, we have $\dist{\altxvar-\yvari{\spongeslot}}+\altpricei{\spongeslot}=\dist{\altxvar-\yvari{\spongeslot}}+\pricei{\spongeslot} \leq 
    \dist{\altxvar-\yvari{\ell}}+\pricei{\ell} \leq 
    \dist{\altxvar-\yvari{\ell}}+\altpricei{\ell} $ and therefore~$\altxvar\in\fullIntervali{\spongeslot}{\altprofile}$. 
        It follows from \cref{lemma:monotonicity} that $\altxvar\leq\xvar$.
    For 
    $1\leq\ell<\spongeslot$,
   using $\altxvar\leq\xvar$ and \cref{fact:increasing}  
   yields $  \dist{\altxvar-\yvari{\ell}}
-\dist{\altxvar-\yvari{\spongeslot}}
\leq \dist{\xvar-\yvari{\ell}} -\dist{\xvar-\yvari{\spongeslot}} $, and it follows that
$
\dist{\altxvar-\yvari{\spongeslot}}+\pricei{\spongeslot} 
\leq
\dist{\altxvar-\yvari{\ell}}+\pricei{\ell} 
\leq
\dist{\xvar-\yvari{\ell}}+\pricei{\ell} 
+ \dist{\altxvar-\yvari{\spongeslot}} - \dist{\xvar-\yvari{\spongeslot}}
$. 
Consequently, 
$ 
\dist{\xvar-\yvari{\altaltslot}}+\pricei{\altaltslot} 
=
\dist{\xvar-\yvari{\altaltslot}}+\altpricei{\altaltslot} 
\leq
\min(0,\dist{\xvar-\yvari{\spongeslot}}+\altpricei{\spongeslot})
=
\min(0,\dist{\altxvar-\yvari{\spongeslot}}+\pricei{\spongeslot} +\dist{\xvar-\yvari{\spongeslot}} - \dist{\altxvar-\yvari{\spongeslot}})
\leq
\min(0,\dist{\xvar-\yvari{\ell}}+\pricei{\ell})
$.
For 
$\spongeslot\leq\ell\leq n$,
we have 
$ 
\dist{\xvar-\yvari{\altaltslot}}+\pricei{\altaltslot} 
=
\dist{\xvar-\yvari{\altaltslot}}+\altpricei{\altaltslot} 
\leq
\min(0,\dist{\xvar-\yvari{\ell}}+\altpricei{\ell})
=
\min(0,\dist{\xvar-\yvari{\ell}}+\pricei{\ell}) 
$. Hence, 
$\xvar\in\Intervali{\altaltslot}{\profile}$, and therefore $\Intervali{\altaltslot}{\altprofile}\subseteq\Intervali{\altaltslot}{\profile}$. 
\qed\end{proof}

\smallskip

\newcommand{\coefficientdist}[2]{C}
\newcommand{\altcoefficientdist}[2]{C'}
\newcommand{\deltabounddist}[2]{\zeta}

The next lemma claims the existence of a feasible solution of~\eqref{upper-level} at distance~$O(\delta^{1/4})$ from an optimal solution of~\eqref{upper-level-alter-delta}.
\cref{lemma:solutiongap} relies on \cref{lemma:sponge-slot,lemma:tool}, and it is the basis of our proof to \cref{prop:upper_bound}.

\begin{lemma}\label{lemma:solutiongap}
There exist 
$\coefficientdist{\distfunction}{\yvari{1:n}},\deltabounddist{\distfunction}{\yvari{1:n}}>0$, 
and maps $\profile,\altprofile \colon (0,\deltabounddist{\distfunction}{\yvari{1:n}}] 
\to \Real^n$ such that 
$\profile(\delta)\in 
\big(
[\pricemindelta{\delta}-\delta,\pricemax+\delta]
\cap
\delta\Integer\big)^n
$ is an optimal solution of~\eqref{upper-level-alter-delta}, $\altprofile(\delta)$ is a feasible solution of~\eqref{upper-level}, and we have $\|\profile(\delta)-\altprofile(\delta)\|_\infty\leq\coefficientdist{\distfunction}{\yvari{1:n}} \sqrt[4]\delta$.
\end{lemma}

We emphasize that the proof of \cref{lemma:solutiongap} will make clear that the constants~$\coefficientdist{\distfunction}{\yvari{1:n}}$ and~$\deltabounddist{\distfunction}{\yvari{1:n}}$ involved in the lemma
only depend on the function~$\distfunction$, the support of~$\xmeasurefunction$ and its lower and upper bounds, and the slot times~$\yvari{1},\dots,\yvari{n}$, and are independent of the values of~$\xmeasurefunction$ and~$\ymeasurefunction$.

\begin{proof}[
\cref{lemma:solutiongap}]
\newcommand{\Deltai}[1]{\Delta_{#1}}%
\newcommand{\Deltaij}[2]{\Delta^{(#1)}_{#2}}%
\newcommand{\Ccst}{C}
\newcommand{\Ecst}{E}
\newcommand{\Acst}{A}
\newcommand{\Bcst}{B}
\newcommand{\Dcst}{D}
Applying~\cref{lemma:sponge-slot} with parameter value $
\csth 
= 
4\xmeasuremax (1+ \MCwave)^{n-1}\sqrt{2 /\stronglyconvex}
$,
where $\MCwave = 
\xmeasuremax\MCtool\lipschitz$ and~$\MCtool$ is specified by \cref{lemma:tool}, yields
specific constants $
 \cstDelta{\csth},\deltabound{\csth}>0$. We show the lemma for $ 
 \coefficientdist{\distfunction}{\yvari{1:n}}
 = 
 2\big(
3 +  (1+ \MCwave)^{n-1} \big)(2\lipschitz^2\stronglyconvex)^{-3/4} + 
\cstDelta{\csth}
$, $\deltabounddist{\distfunction}{\yvari{1:n}} = \deltabound{\csth} $, and for maps $\profile,\altprofile$ that we specify as follows. In view of \cref{prop:bounded-delta}, we set~$\profile(\delta)$ to any discrete optimal solution of~\eqref{upper-level-alter-delta} in~$\big(
[\pricemindelta{\delta}-\delta,\pricemax+\delta]
\cap\delta\Integer
\big)^n$.
We are going to construct $\altprofile(\delta)$ based on $\profile(\delta)$,
for the range $
\delta\in(0,\deltabound{\csth} ]
$ in which~\cref{lemma:sponge-slot} applies.
To lighten the text, we drop the argument~$\delta$ and write $\profile\equiv\profile(\delta)$.
Set~$\MCspongefunction= \big\lceil 
\cstDelta{\csth} \delta^{-{3}/{4}}\big\rceil$, $\altaltprofile = \profile+\MCspongefunction\delta$,
and $\Delta= 4\xmeasuremax\sqrt{2\delta /\stronglyconvex} $.
By definition of~$\profile$ and~$\altaltprofile$, we have $
\fullIntervali{\slot}{\profile}
=
\fullIntervali{\slot}{\altaltprofile}
$
and 
$
\sublevelseti{\slot}{\pricei{\slot}}
\supseteq
\sublevelseti{\slot}{\altaltpricei{\slot}}
$ for $\slot\in[n]$. 
It follows from~\eqref{intersections:a} 
that 
$
\Intervali{\slot}{\profile}
\supseteq
\Intervali{\slot}{\altaltprofile}
$ 
and
$
\xmeasure{
\Intervali{\slot}{\profile}
}
\geq
\xmeasure{
\Intervali{\slot}{\altaltprofile}
}
$ 
for $\slot\in[n]$.
Since~$
\profile
\in(\delta\Integer)^n$ and~$\profile$ is an optimal solution of~\eqref{upper-level-alter-delta}, we get, for $\slot\in[n]$,   
\begin{equation}\label{doubleinequality:both}
\xMeasure{\Intervali{j}{\altaltprofile}} 
\leq
\xMeasure{\Intervali{\slot}{\profile}}
=
\xMeasure{
\Intervali{\slot}{\profilemdelta{\delta}}
}
\leq
\ymeasurei{\slot} + \xmeasuremax\lipschitz \delta 
+ 2\xmeasuremax\sqrt{2 \delta /\stronglyconvex}  
\leq
\ymeasurei{\slot} 
+ \Delta  
\, ,
\end{equation}
where we use 
$\delta\leq\sqrt{\deltamax\delta}
=
(2/\lipschitz)\sqrt{2 \delta/\stronglyconvex}
$.

\smallskip

First suppose that 
$
\sum_{\slot=1}^{n}
\xmeasure{\Intervali{\slot}{\altaltprofile}}
=
0
$.
 Then, we have $
\xmeasure{\Intervali{\slot}{\altaltprofile}}
=
0
$ 
for all~$\slot\in[n]$. Hence, $\altprofile=\altaltprofile
$ is a feasible price profile of~\eqref{upper-level}.

\smallskip

Otherwise, according to \cref{lemma:sponge-slot}, there is a time slot~$\spongeslot\in[n]$ such that 
\begin{equation}\label{spongeassumption}
\xmeasure{\Intervali{\spongeslot}{\altaltprofile}}
+
 \csth \sqrt{\delta}
 \leq
\xMeasure{\Intervali{\spongeslot}{\profile}}  
=
\xMeasure{\Intervali{\spongeslot}{\profilemdelta{\delta}}}  
\refereq{\eqref{doubleinequality:both}}{\leq}
\ymeasurei{\spongeslot} 
+ \Delta  
 \, . 
\end{equation}
From~\eqref{intersections:a} and the first inequality in~\eqref{spongeassumption}, we find
$
\xmeasure{\fullIntervali{\spongeslot}{\profile}} \geq
\xMeasure{\Intervali{\spongeslot}{\profile}}  
\geq
\xmeasure{\Intervali{\spongeslot}{\altaltprofile}}
+  \csth \sqrt{\delta}
>0
$.
 Hence, 
the service load $\xmeasure{\fullIntervali{\spongeslot}{\profile}}$ is a positive quantity.
Now, consider the vectors
\[
\heavisidemvectori{\slot} 
=
(\underbrace{1,\dots,1}_{\slot},0,\dots,0)
\quad\text{and}\quad
\heavisidepvectori{\slot} 
=
(\underbrace{0,\dots,0,1}_{\slot},1,\dots,1)
\quad
\textup{for }
 \slot\in[n]
 \, .
\]
We derive the price profile
$
\altprofile
=
\altaltprofile+ \MCtool \sum_{\slot=1}^{\spongeslot-1}\Deltai{\slot}\heavisidemvectori{\slot}
+
 \MCtool \sum_{\slot=\spongeslot+1}^{n}\Deltai{\slot}\heavisidepvectori{\slot}
$,
where
$\Deltai{\slot} = (1+ \MCwave)^{\slot-1}\Delta$ for $\slot=1,\dots,\spongeslot-1$ and $\Deltai{\slot} = (1+ \MCwave)^{n-\slot}\Delta$ for $\slot=\spongeslot+1,\dots,n$.
We apply $\spongeslot-1$ times 
 \cref{lemma:tool}\ref{lemma:tool:left} with successive parameters~$\Deltai{1},\dots,\Deltai{\spongeslot-1}$, and then $n-\spongeslot$ times \cref{lemma:tool}\ref{lemma:tool:right}  with successive parameters~$\Deltai{n},\dots,\Deltai{\spongeslot+1}$. 
 This is possible because 
 $\xmeasure{\fullIntervali{\spongeslot}{\profile}}>0$, 
 $\fullIntervali{\spongeslot}{\altaltprofile} = \fullIntervali{\spongeslot}{\profile}$ (all the prices are translated upward by the same quantity), and by induction the~$n-1$ price profiles successively considered for \cref{lemma:tool} 
 meet this condition as well%
, with prices not lower than~$\pricemin$%
.
Using
$
\sum_{\altaltslot=1}^{\slot-1}\MCwave\Deltai{\altaltslot} 
-
\Deltai{\slot} 
=
-\Delta
$ 
when 
$1 \leq \slot < \spongeslot$,
$
\sum_{\altaltslot=\slot+1}^{n}\MCwave\Deltai{\altaltslot}  
-
\Deltai{\slot} 
=
-\Delta
$
when
$\spongeslot < \slot \leq  n$,
and the inequality
$
\sum_{\altaltslot=1}^{\spongeslot-1}\MCwave\Deltai{\altaltslot} + \sum_{\altaltslot=\spongeslot+1}^{n}\MCwave\Deltai{\altaltslot}
=
(1+ \MCwave)^{\spongeslot-1}\Delta + (1+ \MCwave)^{n-\spongeslot}\Delta-2\Delta 
\leq 
(1+ \MCwave)^{n-1}\Delta - \Delta
=
\csth \sqrt{\delta} - \Delta
$,
we find
\[
\begin{array}{ll}
 \bigxmeasure{\Intervali{\slot}{\altprofile}}
 \leq
 \max\big(0,\xmeasure{\Intervali{\slot}{\altaltprofile}} + \sum_{\altaltslot=1}^{\slot-1}\MCwave\Deltai{\altaltslot} - \Deltai{\slot}\big)
 \refereq{\eqref{doubleinequality:both}}{\leq}
 \ymeasurei{\slot}
&
\text{if }
   1 \leq \slot < \spongeslot \, ,
\\
 \bigxmeasure{\Intervali{\slot}{\altprofile}}
 \leq
 \max\big(0,\xmeasure{\Intervali{\slot}{\altaltprofile}} + \sum_{\altaltslot=\slot+1}^{n}\MCwave\Deltai{\altaltslot} - \Deltai{\slot}\big)
 \refereq{\eqref{doubleinequality:both}}{\leq}
 \ymeasurei{\slot}
\quad&
\text{if }
   \spongeslot < \slot \leq  n 
\, ,
\\
\multicolumn{2}{l}{
\bigxmeasure{\Intervali{\spongeslot}{\altprofile}}
 \leq
 \max\big(0,\xmeasure{\Intervali{\spongeslot}{\altaltprofile}} + \sum_{\altaltslot=1}^{\spongeslot-1}\MCwave\Deltai{\altaltslot} + \sum_{\altaltslot=\spongeslot+1}^{n}\MCwave\Deltai{\altaltslot}\big)
 \refereq{\eqref{spongeassumption}}{\leq}
  \ymeasurei{\spongeslot}
\, .
}
\end{array}
\]
Hence, we have $  \bigxmeasure{\Intervali{\slot}{\altprofile}}
 \leq
 \ymeasurei{\slot}
$
for all $\slot\in[n]$, and~$\altprofile$ is a feasible price profile of~\eqref{upper-level}.

\smallskip

In any case, $\altprofile(\delta)\equiv\altprofile$ is a feasible price profile of~\eqref{upper-level}. 
Given $\delta\in(0,\deltabound{\csth}]$, we find 
$
\delta\leq\sqrt[4]{(\deltamax)^{3}\delta}
 =
8 \sqrt[4]{\delta/(2\lipschitz^2\stronglyconvex)^3}
$, $
\sqrt{\delta}\leq\sqrt[4]{\deltamax\delta}
=
2\sqrt[4]{\delta/(2\lipschitz^2\stronglyconvex)}
$ and, after computations, 
$ 
\|\profile(\delta)-\altprofile(\delta) \|_\infty 
\leq
\big[8 [2(1+ \MCwave)^{n-1}-1] (2\lipschitz^2\stronglyconvex)^{-3/4} + 
\cstDelta{\csth} \big] \delta^{{1}/{4}}
$,
which completes the proof.
\qed\end{proof}

\smallskip
We are now in a position to show \cref{prop:upper_bound}.

\begin{proof}[
\cref{prop:upper_bound}]
\newcommand{\profiledistance}{z}
Using \cref{lemma:solutiongap}, we derive constants $\coefficientdist{\distfunction}{\yvari{1:n}}$, $\deltabounddist{\distfunction}{\yvari{1:n}}$, 
and maps $\profile,\altprofile \colon (0,\deltabounddist{\distfunction}{\yvari{1:n}}] 
\to \Real^n$ 
such that $\profile(\delta)\in 
\big(
[\pricemindelta{\delta}-\delta,\pricemax+\delta]
\cap
\delta\Integer\big)^n
$ is an optimal solution of~\eqref{upper-level-alter-delta}, $\altprofile(\delta)$ is a feasible solution of~\eqref{upper-level}, and  $
\profiledistance
\coloneqq
\|\profile(\delta)-\altprofile(\delta)\|_\infty\leq\coefficientdist{\distfunction}{\yvari{1:n}} \sqrt[4]{\delta}
$ holds for all~$\delta\in(0,\deltabounddist{\distfunction}{\yvari{1:n}}]$.

Let~$\delta\in(0,\deltabounddist{\distfunction}{\yvari{1:n}}]$.
To lighten the text, we drop the argument~$\delta$ and write $\profile\equiv\profile(\delta)$ and $\altprofile\equiv\altprofile(\delta)$.
We compute the distance between~$\profit{\altprofile}\coloneq \sum_{\altaltslot=1}^n \altpricei{\altaltslot} \mu(\Intervali{\altaltslot}{\altprofile})$, which defines the revenue induced by~$\altprofile$ in~\eqref{upper-level}, 
and the upper bound $\profitdelta{\delta}{\profile}+n\xmeasuremax\lipschitz \delta\pricemax$, where
$\profitdelta{\delta}{\profile} := \sum_{\altaltslot=1}^n
(\pricemdeltai{\delta}{\altaltslot}+\delta)
\xmeasure{
\Intervali{\altaltslot}{\profilemdelta{\delta}}
} = \OPTdelta{\delta}{\Real}$ defines the revenue induced by~$\profile$ in problem~\eqref{upper-level-alter-delta}.
We know from~\eqref{intersections:a} that $
\Intervali{\altaltslot}{\profile} \symdif \Intervali{\altaltslot}{\altprofile}
\subseteq
\big(\fullIntervali{\altaltslot}{\profile}\symdif\fullIntervali{\altaltslot}{\altprofile}\big)  \cup \big(\sublevelseti{\altaltslot}{\pricei{\altaltslot}}\symdif \sublevelseti{\altaltslot}{\altpricei{\altaltslot}}\big)
$.
Using \cref{lemma:twolemmasinone} with the parameter values~$\eta=\eta'=\profiledistance$, we find
$
|\fullIntervali{j}{\altprofile}\setminus\fullIntervali{j}{\profile}|
\leq
2\profiledistance\lipschitz
$ and $
|\fullIntervali{j}{\profile}\setminus \fullIntervali{j}{\altprofile}| \leq 2\profiledistance\lipschitz
$, which imply
 $| \fullIntervali{j}{\profile} \symdif \fullIntervali{j}{\altprofile} | = |\fullIntervali{j}{\altprofile}\setminus \fullIntervali{j}{\profile}| + |\fullIntervali{j}{\profile}\setminus \fullIntervali{j}{\altprofile}|\leq 4  \profiledistance\lipschitz$.
From \cref{lemma:sublevelsets}, 
we also have $
| \sublevelseti{\altaltslot}{\pricei{\altaltslot}} \symdif \sublevelseti{\altaltslot}{\altpricei{\altaltslot}} | 
\leq 
2\sqrt{2 /\stronglyconvex}|\pricei{\altaltslot}-\altpricei{\altaltslot}|^{1 /2}
\leq
2\sqrt{2\profiledistance /\stronglyconvex}
$.
It follows that
$
| \Intervali{\altaltslot}{\profile} \symdif \Intervali{\altaltslot}{\altprofile} | 
\leq
|\fullIntervali{\altaltslot}{\profile}\symdif\fullIntervali{\altaltslot}{\altprofile}
|+|
\sublevelseti{\altaltslot}{\pricei{\altaltslot}}\symdif \sublevelseti{\altaltslot}{\altpricei{\altaltslot}}
|
\leq 
4 \profiledistance\lipschitz
+
2\sqrt{2\profiledistance /\stronglyconvex}
$. Hence,
$\xmeasure{\Intervali{\altaltslot}{\profile}}
\leq
\xmeasure{\Intervali{\altaltslot}{\altprofile}}
+
\xmeasure{\Intervali{\altaltslot}{\profile}\symdif\Intervali{\altaltslot}{\altprofile}}
\leq 
\xmeasure{\Intervali{\altaltslot}{\altprofile}}
+
4 \xmeasuremax\profiledistance\lipschitz
+
2\xmeasuremax\sqrt{2\profiledistance /\stronglyconvex}
$
holds for~$\altaltslot\in[n]$.
Using $\profile\in(\delta\Integer)^n$, we find
\begin{equation*}\label{profitdelta}
\nocolsep
\begin{array}{ll}
 \profitdelta{\delta}{\profile} 
 \hspace{-5mm}&\hspace{5mm}
=
\sum_{\altaltslot=1}^n (\pricei{\altaltslot} + \delta)  \bigxmeasure{\Intervali{\altaltslot}{\profile}}
 \\&
\leq
\sum_{\altaltslot=1}^n 
(\pricei{\altaltslot} + \delta)
\big( \xmeasure{\Intervali{\altaltslot}{\altprofile}}
+
4 \xmeasuremax \profiledistance\lipschitz
+
2\xmeasuremax\sqrt{2\profiledistance /\stronglyconvex}
 \big)
\\&
=
\sum_{\altaltslot=1}^n 
(\altpricei{\altaltslot}+(\pricei{\altaltslot} -\altpricei{\altaltslot}) + \delta) \xmeasure{\Intervali{\altaltslot}{\altprofile}}
+
2\xmeasuremax
\big( 
2 \profiledistance\lipschitz
+
\sqrt{2\profiledistance /\stronglyconvex}
 \big)
\sum_{\altaltslot=1}^n 
(\pricei{\altaltslot} + \delta)
\ \\& 
=
\profit{\altprofile} 
+
\sum_{\altaltslot=1}^n 
((\pricei{\altaltslot} -\altpricei{\altaltslot}) + \delta) \xmeasure{\Intervali{\altaltslot}{\altprofile}}
\\&\hfill
+
2\xmeasuremax
\big( 
2 \profiledistance\lipschitz
+
\sqrt{2\profiledistance /\stronglyconvex}
 \big)
\big(n\delta+ \sum_{\altaltslot=1}^n 
\pricei{\altaltslot} \big) 
\, ,
\end{array}
\end{equation*}
where $\pricei{\altaltslot}\leq\pricemax+\delta$ for all~$\altaltslot\in[n]$.
For all~$\delta\in(0,\deltabounddist{\distfunction}{\yvari{1:n}}]$, we thus find the inequality
$
 \profitdelta{\delta}{\profile(\delta)} 
+n\xmeasuremax\lipschitz \delta\pricemax
\leq
\profit{\altprofile(\delta)} 
+
\vanishing{\delta}
$, 
where the last term is given by 
$
\vanishing{\delta}
=
n \xmeasure{\Real} (\profiledistance + \delta) 
+
2n\xmeasuremax
\big( 
2 \profiledistance\lipschitz
+
\sqrt{2\profiledistance /\stronglyconvex}
 \big)
(\pricemax + 2 \delta)
+n\xmeasuremax\lipschitz \delta\pricemax
$.
It follows from
$
\profiledistance
\leq
\coefficientdist{\distfunction}{\yvari{1:n}} \sqrt[4]{\delta}
$
that $ \vanishing{\delta} =O(\delta^{1/8}) $.
Finally, we find
$ 
\OPTdelta{\delta}{\Real}
+n\xmeasuremax\lipschitz \delta\pricemax
=
\profitdelta{\delta}{\profile(\delta)} 
+n\xmeasuremax\lipschitz \delta\pricemax
\leq
\profit{\altprofile(\delta)} 
+
\vanishing{\delta}
\leq
\OPT{\delta\Integer}
+
\vanishing{\delta}
$, which completes the proof.%
\qed\end{proof}
%


\newcommand{\types}{\ell}%
\newcommand{\durations}{m}%

\section{Conclusion}

We devise polynomial-time algorithms for an industrial problem of time slot pricing formulated as a bilevel program. In our framework, revenue maximization is reduced to finding longest paths in directed acyclic graphs. The solutions developed in this paper rest on a series of assumptions.

When the set of prices~$P$ is finite, the users' ability to choose a unique ($\xmeasurefunction$-almost everywhere) time slot follows from the strict convexity of the cost function~$\distfunction$---on which \cref{lemma:monotonicity,prop:intervals} directly rely---together with the non-atomic nature of the population~$\xmeasurefunction$. Without these two assumptions, the bilevel problem becomes ambiguous as to which time slots indecisive users may choose. An optimistic approach then is to introduce additional incentives to assist the indecisive users in their final decisions. A more conservative approach is to stick to the assumption that each user may choose any time slot that minimizes their total cost, in which case the sets $\fullIntervali{1}{\profile},\dots,\fullIntervali{n}{\profile}$ remain ordered as in \cref{prop:intervals} but cease to have disjoint interiors. Relaxing these two assumptions would constitute a natural extension of the present work. How the results extend to this setting remains to be clarified.

When~$P$ is~$\Real$, lower bounds for the optimal revenue are obtained under the same assumptions by solving~\eqref{upper-level} over discretized~$P$. Strong convexity of~$\xmeasurefunction$ allows us to derive upper bounds in a similar way by solving an approximate problem~\eqref{upper-level-alter-delta}. Asymptotical exactness of the bounds and polynomial convergence in the discretization step follow under the additional assumption that~$\xmeasurefunction$ is interval-supported and its density is bounded above zero.

Key aspects of the graphical approach taken in this work are simplicity of use, depth of analysis, and flexbility. 
Refinements can be incorporated into the basic model to reflect user diversity, for instance in cost function or in service duration---some of these refinements translating into simple operations on graphs. We leave open the question of computational complexity in more elaborate settings of the time slot pricing problem.


\begin{acknowledgements}
We are grateful to our colleagues at EDF R\&D, Guilhem Dupuis and Cheng Wan, for many stimulating discussions that played a key role in identifying the problem and clarifying its industrial relevance. This research benefited from the support of the FMJH Program Gaspard Monge for optimization and operations research and their interactions with data science.
No new data were created or analyzed in this study. Data sharing is not applicable to this article.
\end{acknowledgements}

\bibliographystyle{spmpsci}
\urlstyle{same} 
\bibliography{ENPC}

\end{document}